\pdfoutput=1
\documentclass[a4paper,11pt]{amsart}
\usepackage[latin1]{inputenc}
\usepackage{amsmath, amsthm, amssymb, amsfonts, amscd}
\usepackage[english]{babel}
\usepackage[all]{xy}
\usepackage[linktocpage=true]{hyperref}
\usepackage{color}
\usepackage[table]{xcolor}
\usepackage{pinlabel}
\usepackage{tikz-cd}
 \usepackage{booktabs}
\usepackage{enumitem}
\usepackage{subcaption}
\usepackage{float}

\newtheorem{thm}{Theorem}[section]
\newtheorem*{thm*}{Definition}
\newtheorem*{thmm*}{Theorem}

\newtheorem{cor}[thm]{Corollary}
\newtheorem{defi}[thm]{Definition}

\theoremstyle{remark}
\newtheorem{ex}[thm]{Example}
\newtheorem{rmk}[thm]{Remark}
\definecolor{orange}{RGB}{255,127,0}
\newcommand{\red}{\textcolor{red}}
\newcommand{\blue}{\textcolor{blue}}

\newcommand{\grey}{\textcolor{orange}}
\numberwithin{figure}{section}
\numberwithin{equation}{section}
\newcommand\restr[2]{{% we make the whole thing an ordinary symbol
  \left.\kern-\nulldelimiterspace % automatically resize the bar with \right
  #1 % the function
  \vphantom{\big|} % pretend it's a little taller at normal size
  \right|_{#2} % this is the delimiter
  }}

\title{$f$-distance of knotoids and protein structure}
\author[A. Barbensi, D. Goundaroulis]{Agnese Barbensi, Dimos Goundaroulis}

\address{AB: Mathematical Institute, University of Oxford, Oxford, UK.}

\address{DG: The Center for Genome Architecture, Baylor College of Medicine, Houston, TX 77030, USA and Department of Molecular and Human Genetics, Baylor College of Medicine,
Houston TX 77030, USA}

\begin{document}

\begin{abstract}
Recent studies classify the topology of proteins by analysing the distribution of their projections using knotoids. The approximation of this distribution depends on the number of projection directions that are sampled. Here we investigate the relation between knotoids differing only by small perturbations of the direction of projection. Since such knotoids are connected by at most a single forbidden move, we characterise forbidden moves in terms of equivariant band attachment between strongly invertible knots and of strand passages between $\theta$-curves. This allows for the determination of the optimal sample size needed to produce a well approximated knotoid distribution. Based on that and on topological properties of the distribution, we probe the depth of knotted proteins with the trefoil as the predominant knot type without using subchain analysis.

\vspace{3mm}
\end{abstract}

\maketitle

\section{Introduction}
Knotoids provide a generalisation of knots that deals with the problem of classifying knottiness for open curves \cite{turaev} . They are defined as equivalence classes of diagrams of open arcs, up to isotopies of $S^2$ and Reidemeister moves performed away from the endpoints. Each equivalence class forms a specific knotoid type. Some examples of knotoid diagrams can be seen in Figure \ref{fig:symmetries}.

\begin{figure}[h]
\includegraphics[width=10cm]{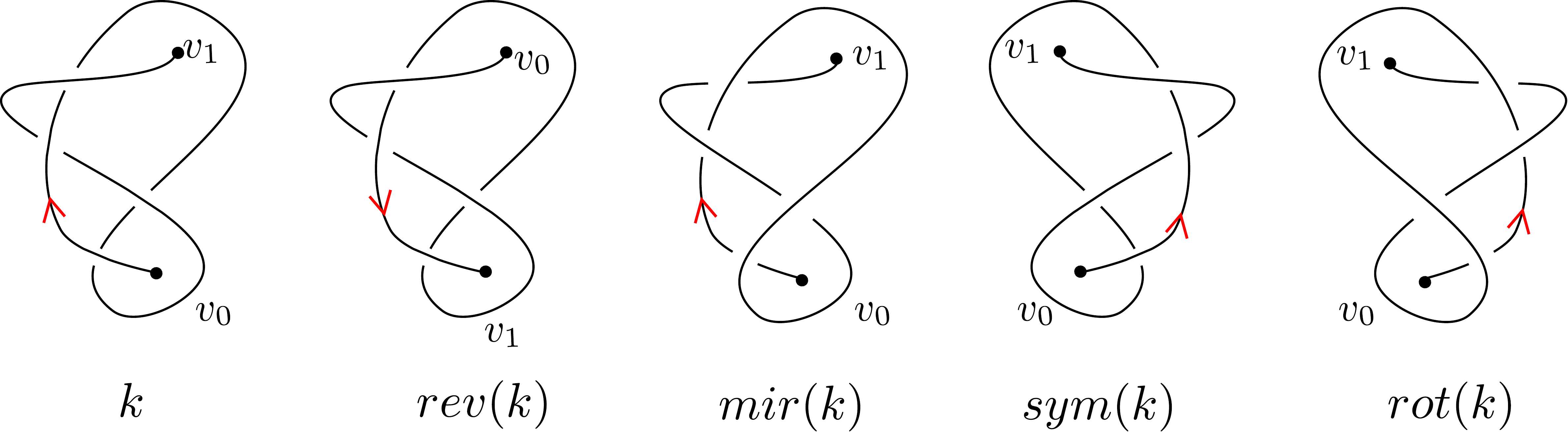}
\caption{From left to right, a knotoid $k$, its reverse $rev(k)$, its mirror reflection $mir(k)$, its symmetric $sym(k)$ and its rotation $rot(k)$. Rotation, reversion, mirror reflection and symmetry are involutive operations on knotoids. These are defined at the beginning of Section \ref{sec:two}.}
\label{fig:symmetries}
\end{figure}

In the past few years, knotoids have been used to classify entanglement in proteins \cite{goundaroulis2017topological, goundaroulis2017, knotoID, knotprot}. Proteins are long chains of amino acids that fold into specific conformations that can sometimes form open ended knots. The fraction of knotted proteins is fairly small \cite{knotprot}, and even if the presence of knots in proteins slows the folding process \cite{mallam2012knot, dabrowski2015prediction, sulkowska2012energy}, it is known that the knotted domains of some families of proteins have been conserved through evolution \cite{sulkowska2012conservation}. While studies seem to suggest that knots provide advantages to some proteins \cite{soler2013effects, soler2014effects, dabrowski2016, san2017knots}, the biological purpose of the presence of knots in proteins is still an open and interesting question in biology. For this reason, understanding the topological features of knotted proteins is an important step in investigating the effect of the presence of knots to structure and function of proteins.  With the knotoids approach, a protein is represented as an open-ended polygonal chain and it is studied by considering all different projections of it. Subsequently, these projections are analysed as knotoids. The topology of the curve is characterised by a distribution of knotoid types, also called the spectrum of the curve, while the predominate type is called the \emph{predominate knotoid}. Since considering all possible projections of a curve is not computationally feasible, we usually sample from the knotoid distribution.

In this work, we study the relation between pairs of knotoids that are obtained from projections that differ from one another only by a small perturbation. Indeed, small perturbations in the choice of the direction of projections can either leave the corresponding knotoid type unchanged (\emph{i.e.} by changing the knotoid diagram by isotopies of $S^2$) or have the effect of performing a sequence of, so-called, \emph{forbidden moves} (see Figure \ref{fig:forbidden}) on the knotoid diagram. Thus, when the spectrum approximates well the knotoid distribution, a pair of different knotoids whose projections are related by a small perturbation will differ by a single forbidden move at most.

\begin{figure}[h]
\includegraphics[width=10cm]{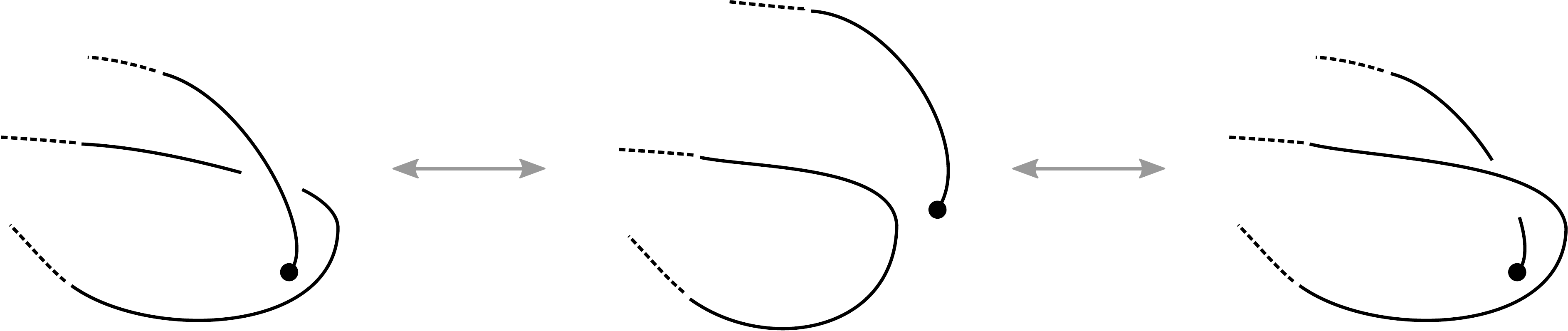}
\caption{Forbidden moves between knotoid diagrams. Performing a forbidden move on a knotoid diagram might result in changing the knotoid type. Moreover, any knotoid diagram can be transformed into the trivial one by a finite sequence of forbidden moves.}
\label{fig:forbidden}
\end{figure}

In analogy to the case of knots and crossing changes, we define a distance on knotoids using forbidden moves. 

\begin{thm*}
Given two knotoids $k_1$ and $k_2$, their \emph{forbidden move-distance} or $f$-distance $d_f(k_1, k_2)$ is the minimal number of forbidden moves, across all representatives of $k_1$ and $k_2$, needed to transform $k_1$ into $k_2$. 
\end{thm*}

Note that since every knotoid diagram can be unknotted by a finite sequence of forbidden moves \cite{turaev,dbc}, the forbidden move-distance is indeed a well defined distance on knotoids. 

We characterise forbidden moves on knotoids in terms of equivariant band attachments on strongly invertible knots, using the correspondence between knotoids and strongly invertible knots proved in \cite{dbc}, and in terms of crossing changes on $\theta$-curves. The main theorem of this paper is the following:

\begin{thmm*}
Consider two equivalence classes of of knotoids $k_1$ and $k_2$ up to rotation and inversion (see Figure \ref{fig:symmetries}). The following are equivalent:
\begin{enumerate}
\item[(i)]  $k_1$ and $k_2$ differ by a single forbidden move; 
 \item[(ii)] their corresponding $\theta$-curves $t_{\approx}(k_1)$ and $t_{\approx}(k_2)$ differ by a strand passage of the edge $e_0$ over either $e_+$ or $e_-$;
 \item[(iii)] their corresponding strongly invertible knots $\gamma_S(k_1)$ and $\gamma_S(k_2)$ differ by an equivariant band surgery.
\end{enumerate}

\end{thmm*}

This result allows us to produce lower bounds on the $f$-distance between knotoids. We then compute the total number of strand passages on all the knotoid diagrams with up to $6$ crossings. This computation gives us upper bounds for the $f$-distance between knotoids with up to $6$ crossings. By comparing lower and upper bounds we compute the $f$-distances $d_f(k_1,k_2)$ for each pair of knotoids $k_1$ and $k_2$ with minimal crossing number $\leq 4$. We then create the \emph{$f$-distance table} for knotoids with minimal crossing number $\leq 4$. 

In the second part of this work we apply the main theorem and the table of $f$-distances  to determine an optimal size for the set of sampled projections that not only approximates well the knotoid distribution but also favours computational speed. 

Based on that, we probe the depth of knotted proteins having the trefoil as predominate knotoid, without analysing the set of all of its subchains. A protein is often considered deeply knotted when at least 20 amino acids can be removed starting from either terminus of the protein structure before converting it to a different type of knot (possibly the trivial one) \cite{taylor2003protein}. The knotted core of a protein is the shortest subchain before the knot type conversion mentioned before and it is believed to have an important biological role in the protein's function \cite{mallam2007, christian}. Recent studies show that formation of deep knots with characteristic structural motifs provides a favorable environment for active sites in enzymes \cite{dabrowski2016}. The depth and, consequently, the length of the knotted core of a protein are currently determined by computationally expensive subchain analysis \cite{knotoID, knotprot}.  

With this work, we provide evidence that the knotoid spectrum of a knotted protein contains important information on the geometry and topology of the protein itself. 

\section*{Notation} Throughout this paper knotoids are indicated according to the tabulation created by the second author in \cite{tableknotoids}. All maps and manifolds are assumed to be smooth, and for maps and sets we will use the notation of \cite{dbc}, namely:
\begin{itemize}
 \item $\mathbb{K}(S^2)$ and $\mathbb{K}(S^2)/_{\approx}$ are the sets of knotoids and the set of knotoids up to rotation and inversion;
 \item $\Theta^s$ is the set of simple labelled $\theta$-curves in $S^3$ and $\Theta^s/_\approx$ is the set of simple labelled $\theta$-curves in $S^3$ up to relabelling the vertices, and up to relabelling the vertices and the edges $e_-$ and $e_+$, respectively;
 \item $\mathcal{K}SI(S^3)$ is the set of strongly invertible knots $(K,\tau)$ in $S^3$.
\end{itemize}

\section{On forbidden moves, crossing changes and band surgeries}\label{sec:two}

Like knots, knotoids admit natural commuting involutive operations. These operations are called \emph{reversion}, \emph{mirror reflection}, \emph{symmetry} and \emph{rotation}, see Figure \ref{fig:symmetries}. Reversion has the effect of changing the orientation of a knotoid, and mirror reflection transforms a knotoid into a knotoid represented by the same diagrams with all the crossings changed. Symmetry reflects a knotoid diagram with respect to the line in $\mathbb{R}^2$ passing through the endpoints. The last involution, the rotation, is defined as the composition of symmetry and mirror reflection.

We will sometimes consider knotoids up to these involutions. Indeed, following the notation of \cite{dbc}, we will denote by $\mathbb{K}(S^2)/_\approx$ the set of knotoids in $S^2$ up to rotation and reversion.

\subsection{Knotoids, $\theta$-curves and strongly invertible knots.}

\begin{defi}
A $\theta$\emph{-curve} is a graph embedded in $S^3$ with $2$ vertices, $v_0$ and $v_1$, and $3$ edges, $e_+, e_-$ and $e_0$, each of which joins $v_0$ to $v_1$, taken up to ambient isotopies preserving the labels of the vertices and the edges. 
\end{defi}

Note that the curves $e_0 \cup e_-$, $e_- \cup e_+$ and $e_0 \cup e_+$ form knots, called the \emph{constituent knots} of the $\theta$-curve. A $\theta$-curve is called \emph{simple} if its constituent knot $e_- \cup e_+$ is the trivial knot. An example of a simple $\theta$-curve is shown in Figure \ref{fig:thetaexample}.

\begin{figure}[h]
\includegraphics[width=3cm]{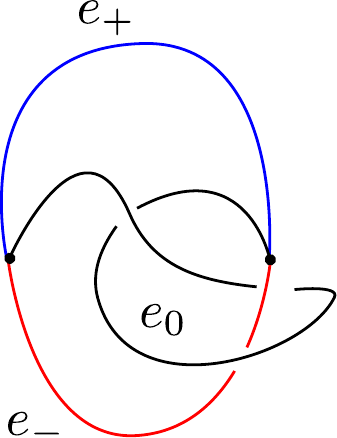}
\caption{A simple $\theta$-curve. The constituent knot $e_- \cup e_+$ (the blue\&red circle) is the trivial knot.}
\label{fig:thetaexample}
\end{figure}

We will find useful to consider $\theta$-curves up to certain particular symmetries. Indeed, following the notation of \cite{dbc}, we will denote by $\Theta^s/_\approx$ the set of simple $\theta$-curves up to relabelling the vertices $v_0$ and $v_1$, and the edges $e_-$ and $e_+$.\\

Recall that the symmetry group of a knot $K$, $Sym^+(S^3,K)$, is the group of diffeomorphisms of the pair $(S^3, K)$ preserving the orientation of $S^3$, where the diffeomorphisms are taken up to isotopies.

\begin{defi}\label{defi:strongly}
A \emph{strongly invertible knot} is a pair $(K, \tau)$, where $K$ is a knot in $S^3$, and $\tau \in Sym^+(S^3,K)$ is called a \emph{strong inversion}, and it is an orientation preserving element of $S^3$ that reverses the orientation of $K$ and such that $\tau^2$ is equal to the identity, taken up to conjugacy in $Sym^+(S^3,K)$. 
\end{defi}

We will denote by $fix(\tau)$ the fixed point set of $\tau$, which is the set of points $x$ suche that $\tau(x) = x$. Note that the positive solution of the Smith Conjecture (see \emph{e.g} \cite{wald}) implies that $fix(\tau)$ in our case is a trivial knot intersecting $K$ in two points. We will denote by $\mathcal{K}SI(S^3)$ the set of strongly invertible knots. As an example, the trefoil knot $3_1$ admits such an inversion in its symmetric group, see Figure \ref{fig:siknot}.

\begin{figure}[h]
\includegraphics[width=3cm]{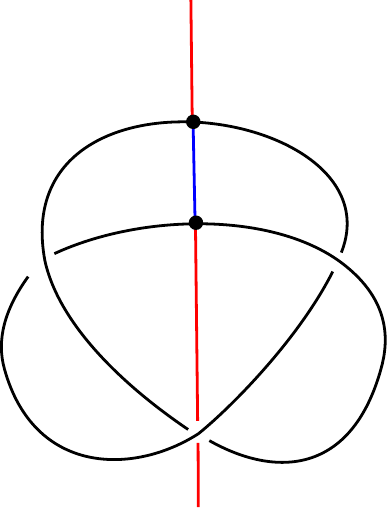}
\caption{The trefoil knot admits a unique strong inversion $\tau$. The fixed point set of $fix(\tau)$ is shown in the picture as a vertical line (the blue\&red line).}
\label{fig:siknot}
\end{figure}

The set $\mathbb{K}(S^2)/_{\approx}$ (\emph{i.e.}~the set of unoriented knotoids in $S^2$, taken up to rotation) is in bijection with the sets $\Theta^s/_\approx$ and $\mathcal{K}SI(S^3)$ (see \cite{turaev} and \cite{dbc} respectively for the proofs): $$t_{\approx}: \mathbb{K}(S^2)/_{\approx} \longrightarrow \Theta^s/_\approx .$$  
$$\gamma_S: \mathbb{K}(S^2)/_{\approx} \longrightarrow \mathcal{K}SI(S^3) .$$

\begin{figure}[h]
\includegraphics[width=7cm]{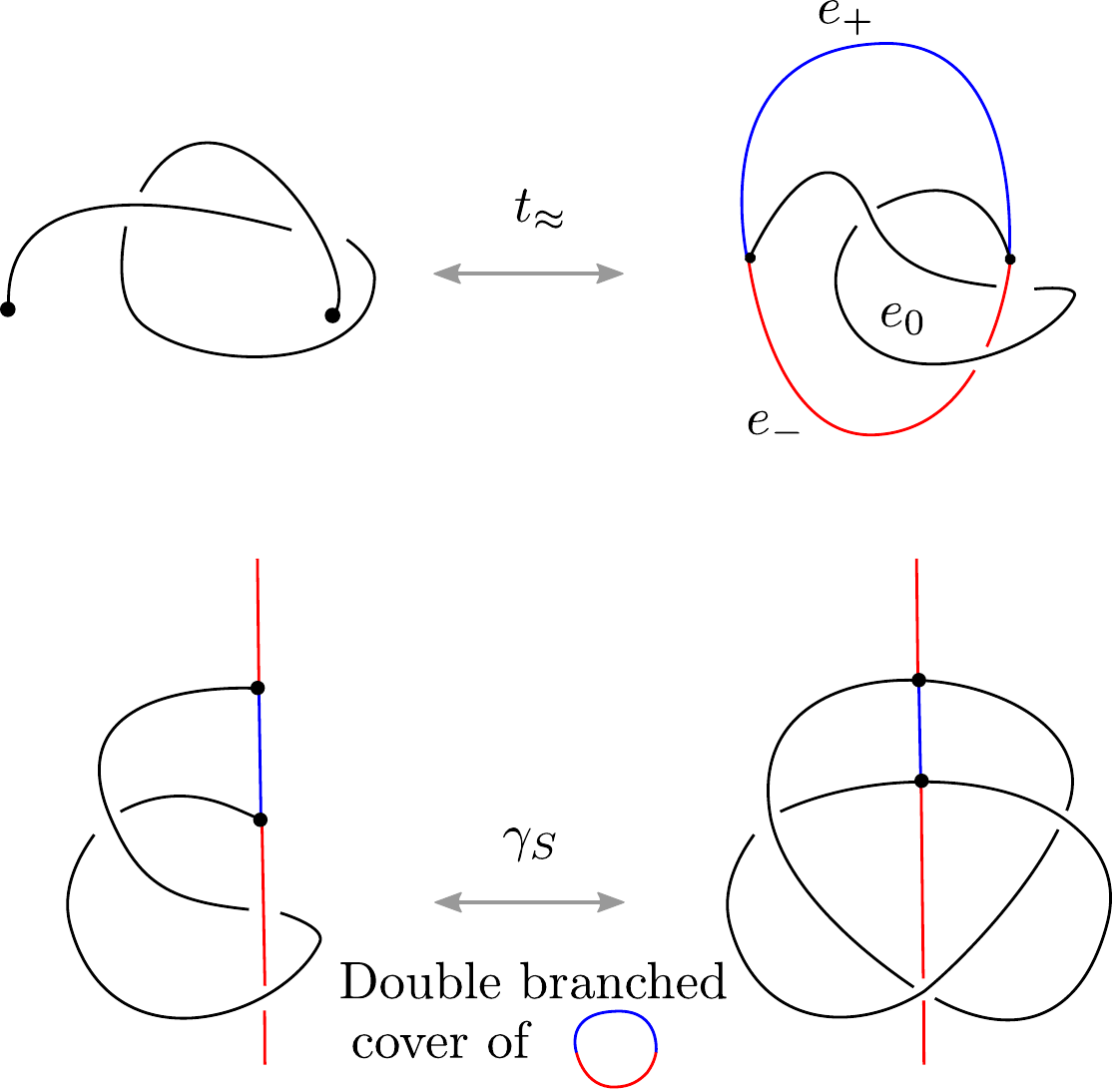}
\caption{On the top, a knotoid and its corresponding simple $\theta$-curve. On the bottom, the associated strongly invertible knot obtained by taking the double cover of $S^3$ branched along $e_- \cup e_+$.}
\label{fig:theta}
\end{figure}

These correspondences work schematically as follows (see also Figure \ref{fig:theta} for an example).\\

\textbf{Step 1.} Given a diagram in $S^2$ representing a knotoid $k$, we construct an embedded arc in $S^2 \times I$ by pushing the overpasses of the diagram into the upper half-space, and the underpasses into the lower one. The endpoints lie in the lines $v_0 \times I$ and $v_1 \times I$. We then obtain a $\theta$-curve by collapsing $S^2 \times \partial I$ to two points. The vertices of the $\theta$-curve are the endpoints of $k$, $e_0 =k$, $e_+$ is the edge containing the image of $S^2 \times \{1\}$ and $e_-$ the one containing the image of $S^2 \times \{-1\}$.\\

\textbf{Step 2.} Given a simple, labelled $\theta$-curve, we obtain a strongly invertible knot by taking the double cover of $S^3$ branched\footnote{For a definition of double branched covers, and for an explanation on how to obtain the double cover of $S^3$ branched along a trivial knot see \emph{e.g.}~\cite{rolfsen}.} along the constituent trivial knot $e_- \cup e_+$, see Figure \ref{fig:theta}.\\

\noindent The converse follows in a similar fashion: \\

\textbf{Step 1.} Given a strongly invertible knot $(K,\tau)$, we label the two halves of the circle $fix(\tau)$ as $e_+$ and $e_-$. The involution $\tau$ induces a projection map $p: S^3 \longrightarrow S^3/\tau \approx S^3$. We then obtain the $\theta$-curve $p(fix(\tau)) \cup p(K)$, where $p(K) = e_0$ and $p(fix(\tau)) = e_- \cup e_+$.\\

\textbf{Step 2.} Given a simple, labelled $\theta$-curve in $S^3$, we can think about it as being embedded in $\mathbb{R}^3$. We can isotope it in such a way that $e_+$ and $e_-$ lie in the upper and lower half-spaces, respectively. We can always do that in such a way that the projection of $e_0$ into $\mathbb{R}^2 \times \{0\}$ is standard. Such projection gives us a knotoid. 

\begin{rmk}
Since two knotoids $k$ and $k_{\textbf{rot}}$ (respectively $rev(k)$) differing by a rotation (respectively inversion) correspond to $\theta$-curves differing by swapping the $e_-$ and $e_+$ labels (respectively $v_0$ and $v_1$), and since the double branched covers of such $\theta$-curves produce equivalent strongly invertible knots, we have the correspondences.
 
\end{rmk}

\subsection{Characterisation of forbidden moves}
In what follows we will give a characterisation of forbidden moves in terms of operations on $\theta$-curves and on strongly invertible knots. 

\subsubsection{Crossing changes on $\theta$-curves}

A forbidden move on $k$ corresponds to performing a strand passage (\emph{i.e.} a crossing change) on the $\theta$-curve $t_\approx (k)$, see Figure \ref{fig:forbiddencrossing}. More precisely, a forbidden move induces a strand passage between the arc $e_0$ and either $e_+$ or $e_-$.

\begin{rmk}\label{rmk:justoneknot}
Call $K^{\pm}_{k}$ the constituent knot of $t_\approx (k)$ given by $e_0 \cup e_{\pm}$. From the previous construction, it follows that
a forbidden move induces a crossing change on exactly one among $K^{+}_{k}$ or $K^{-}_{k}$. In particular, this specific strand passage cannot change simultaneously both these constituent knots of $t_\approx (k)$. 
\end{rmk}

\begin{rmk}\label{rmk:closure}
Note that given a knotoid $k$, the pair $(K^{+}_{k}, K^{-}_{k})$ can be obtained by computing the \emph{overpassing closure} and the mirror image of the \emph{underpassing closure}  of $k$ (for a definition see \cite{turaev}).
 
\end{rmk}

\subsubsection{Band surgeries on strongly invertible knots}
A band surgery is an operation which changes a link into another link.
\begin{defi}
Let $L_1$ be a link and $b: I \times I \longrightarrow S^3$ an embedding such that $L_1 \cap b(I \times I) = b(I \times \partial I)$. The link $L_2 = (L_1 \setminus b(I \times \partial I) ) \cup  b(\partial I \times I)$ is said to be obtained from $L_1$ by a \emph{band surgery} along the band $B=  b(I \times I)$, see Figure \ref{fig:bs}.
\end{defi}

The band surgery is called \emph{coherent} if it respects the orientation of $L_1$ and $L_2$, otherwise it is called \emph{non-coherent}, see Figure \ref{fig:bs}. A non-coherent band surgery it is often called a $H_2$-move (see \emph{e.g.} \cite{band2}). Contrary to the case of coherent band surgeries, $H_2$-moves preserve the number of the components of links. This means that the result of an $H_2$-move performed on a knot will always be a knot.

\begin{figure}[h]
\includegraphics[width=6cm]{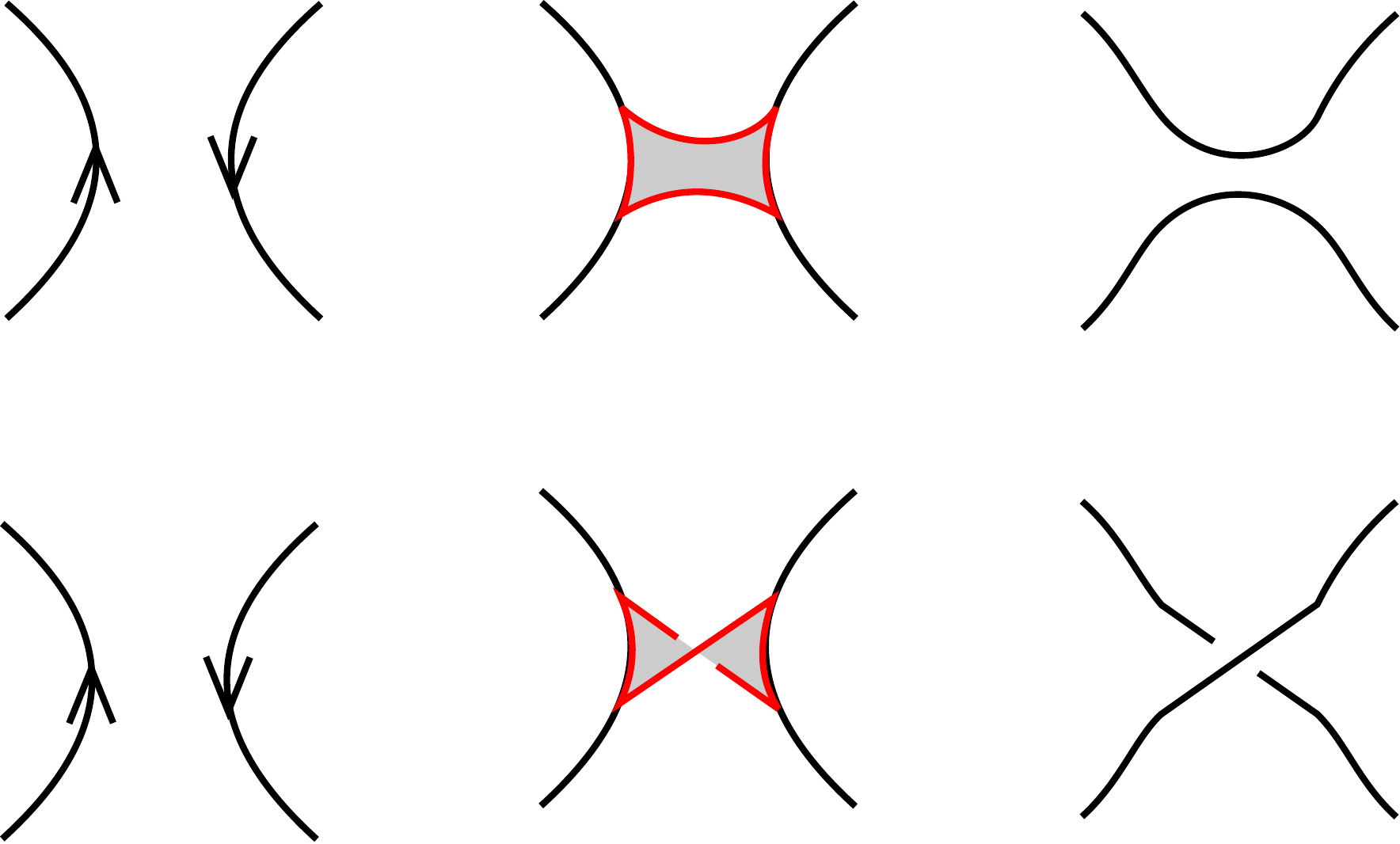}
\caption{Local pictures for a coherent (top) and a non-coherent (bottom) band surgery.}
\label{fig:bs}
\end{figure}

As discussed in \cite{dbc}, two knotoids that differ by a forbidden move have lifts that are related by a single $H_2$-move. Moreover, the band attachment is \emph{equivariant} (as in Definition \ref{equivariantBA} below) with respect to the involutions of the two knots (see Figure \ref{fig:forbiddencrossing}). 

\begin{defi}\label{equivariantBA}
Consider a strongly invertible knot $(K_1, \tau_1)$. We say that the strongly invertible knot $(K_2, \tau_2)$ is obtained from $(K_1, \tau_1)$ by an \emph{equivariant band surgery} if the knots $K_1$ and $K_2$ are related by an $H_2$-move, such that:

\begin{itemize}
 
 \item $fix(\tau_1)$ intersects the band $b(I \times I)$ transversally exactly once in its interior and the band is invariant under $\tau_1$;
 
 \item $(K_2, \tau_2)$ and $(K'_1, \tau_1)$ are equivalent as strongly invertible knots\footnote{Recall from Definition \ref{defi:strongly} that two strongly invertible knots $(K,\tau)$ and $(K',\tau')$ are \emph{equivalent} if $K$ and $K'$ are equivalent as knots in $S^3$, and $\tau$ and $\tau'$ are conjugated in $Sym^+(S^3,K)$. }, where $K'_1$ is the knot obtained from $K_1$ by doing the band surgery.
\end{itemize}

\end{defi}

In other words, an equivariant band surgery between two strongly invertible knots is a band surgery respecting both the involutions of the knots. \\

\begin{figure}[h]
\includegraphics[width=10cm]{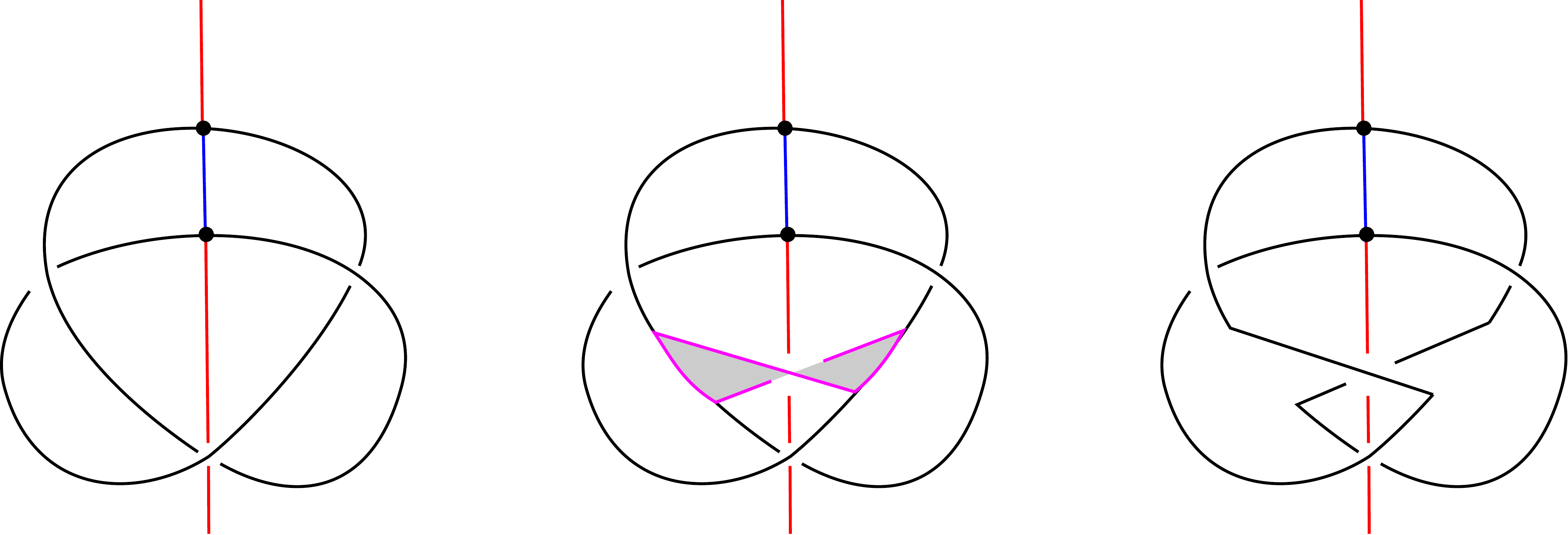}
\caption{The strongly invertible knots $(3_1, \tau)$ and $(0_1, \tau')$ are related by an equivariant band surgery.}
\label{fig:equivariantBS}
\end{figure}

An example of an equivariant band surgery is shown in Figure \ref{fig:equivariantBS}. \\

\begin{figure}[h]
\includegraphics[width=11cm]{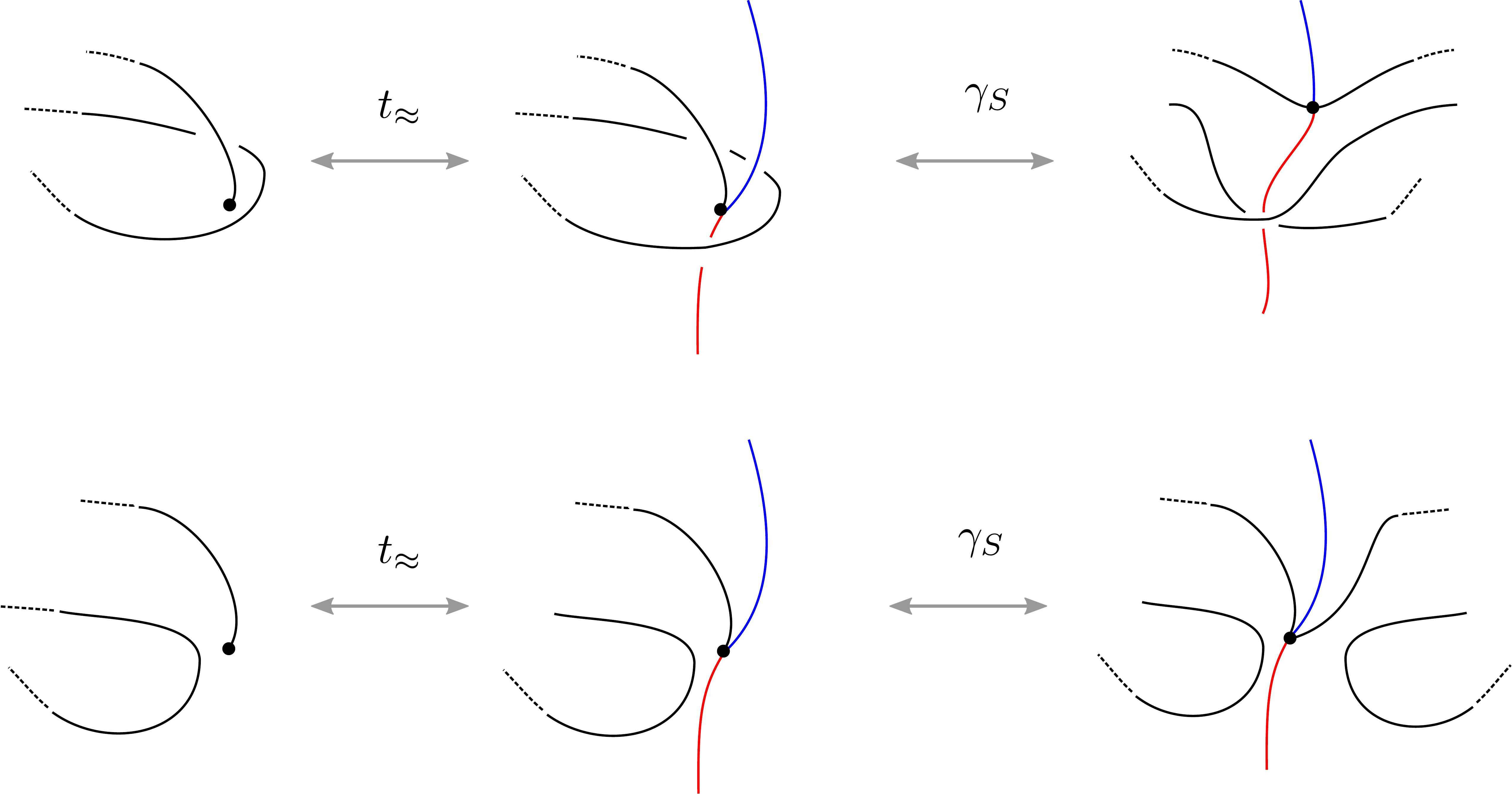}
\caption{A forbidden move between two knotoids $k_1$ and $k_2$ induces a strand passage between the arcs $e_0$ and $e_{\pm}$ between the corresponding $\theta$-curves, and an equivariant band attachment between the corresponding strongly invertible knots.}
\label{fig:forbiddencrossing}
\end{figure}

We are now able to prove our main result.

\begin{thm}\label{characterisation}
Consider two equivalence classes of knotoids $k_1$ and $k_2$ up to rotation and inversion. The following are equivalent:
\begin{itemize}
\item  $k_1$ and $k_2$ differ by a single forbidden move; 
 \item their corresponding $\theta$-curves $t_{\approx}(k_1)$ and $t_{\approx}(k_2)$ differ by a strand passage of the edge $e_0$ over either $e_+$ or $e_-$;
 \item their corresponding strongly invertible knots $\gamma_S(k_1)$ and $\gamma_S(k_2)$ differ by an equivariant band surgery.
\end{itemize}

\end{thm}

\begin{proof}
Thanks to the discussion of the previous subsections, it is enough to show the following.
\begin{itemize}
 \item Given two strongly invertible knots related by an equivariant band attachment, their corresponding $\theta$-curves are related by a strand passage of $e_0$ through either $e_+$ or $e_-$;
 
 \item given two $\theta$-curves related by a strand passage of $e_0$ through either $e_+$ or $e_-$, their corresponding knotoids differ by a forbidden move.
\end{itemize}

Consider then an equivariant band surgery between two strongly invertible knots $(K_1,\tau_1)$ and $(K_2, \tau_2)$. Up to ambient isotopies fixing the circle $fix(\tau_1)$ the band attachment locally looks like the one in the top part Figure \ref{fig:localpicture1}, with possibly the opposite twists on the band. On the quotient $S^3/\tau_1 \approx S^3$ this results in a strand passage between the arcs $e_0$ and one between $e_+$ or $e_-$ in the $\theta$-curve $p(fix(\tau)) \cup p(K)$, as shown in the bottom of Figure \ref{fig:localpicture1}.

\begin{figure}[h]
\includegraphics[width=11cm]{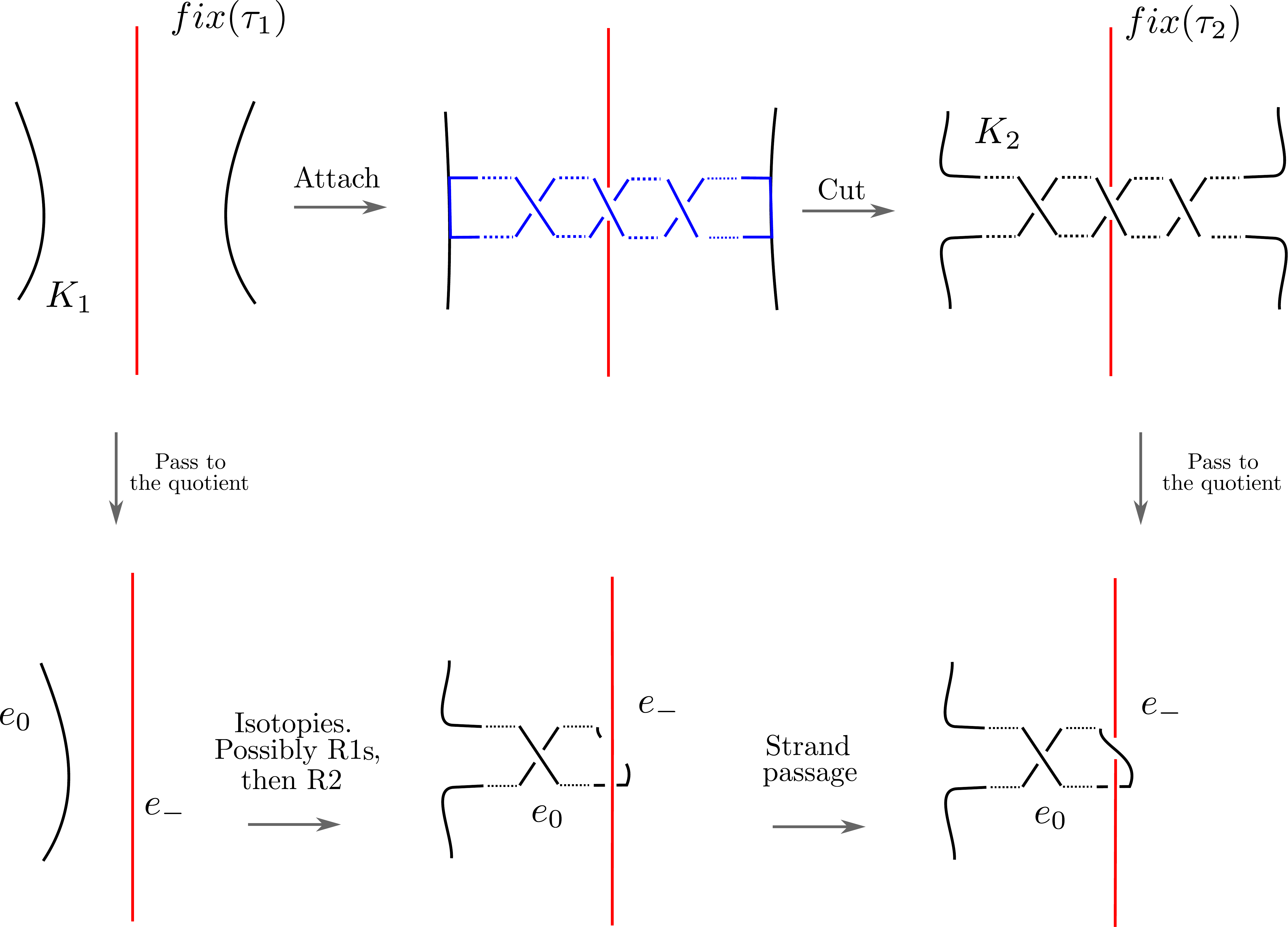}
\caption{On the top row, we present a zoomed-in instance of two strongly invertible knots $(K_1,\tau_1)$ on the left and $(K_2, \tau_2)$ on the right that are related by an equivariant band surgery. The top middle row shows the band attached between opposite arcs of the knot, up to ambient isotopies  fixing the circle $fix(\tau_1)$ (and up to inverting the crossings). The band has an odd number of twists. On the bottom row, the corresponding effect on the associated $\theta$-curves. On the bottom left and right we see the $\theta$-curves obtained by quotienting the knots $(K_1,\tau_1)$ and $(K_2, \tau_2)$ along $\tau_1$ and $\tau_2$. The effect of the band attachment on the $\theta$-curves corresponds (up to ambient isotopies) to a crossing change. More precisely, as shown in the picture, if the attached band has $2n+1$ twists, the $\theta$-curves are related by a sequence of $n$ Reidemeister moves of type I followed by a Reidemeister move of type II and by a single strand passage.}
\label{fig:localpicture1}
\end{figure} 

Analogously, consider a simple $\theta$-curve. Up to label preserving ambient isotopies fixing the circle $e_- \cup e_+$, any strand passage between the arc $e_0$ and the arc $e_{\pm}$ locally looks like the one shown in the top part of Figure \ref{fig:localpicture2} (up to changing the crossing between $e_0$ and $e_{\pm}$). The bottom right part of Figure \ref{fig:localpicture2} shows how this translates into a forbidden move on the corresponding knotoid. The case where the crossing between $e_0$ and $e_{\pm}$ is the opposite one is similar.  

\begin{figure}[h]
\includegraphics[width=10cm]{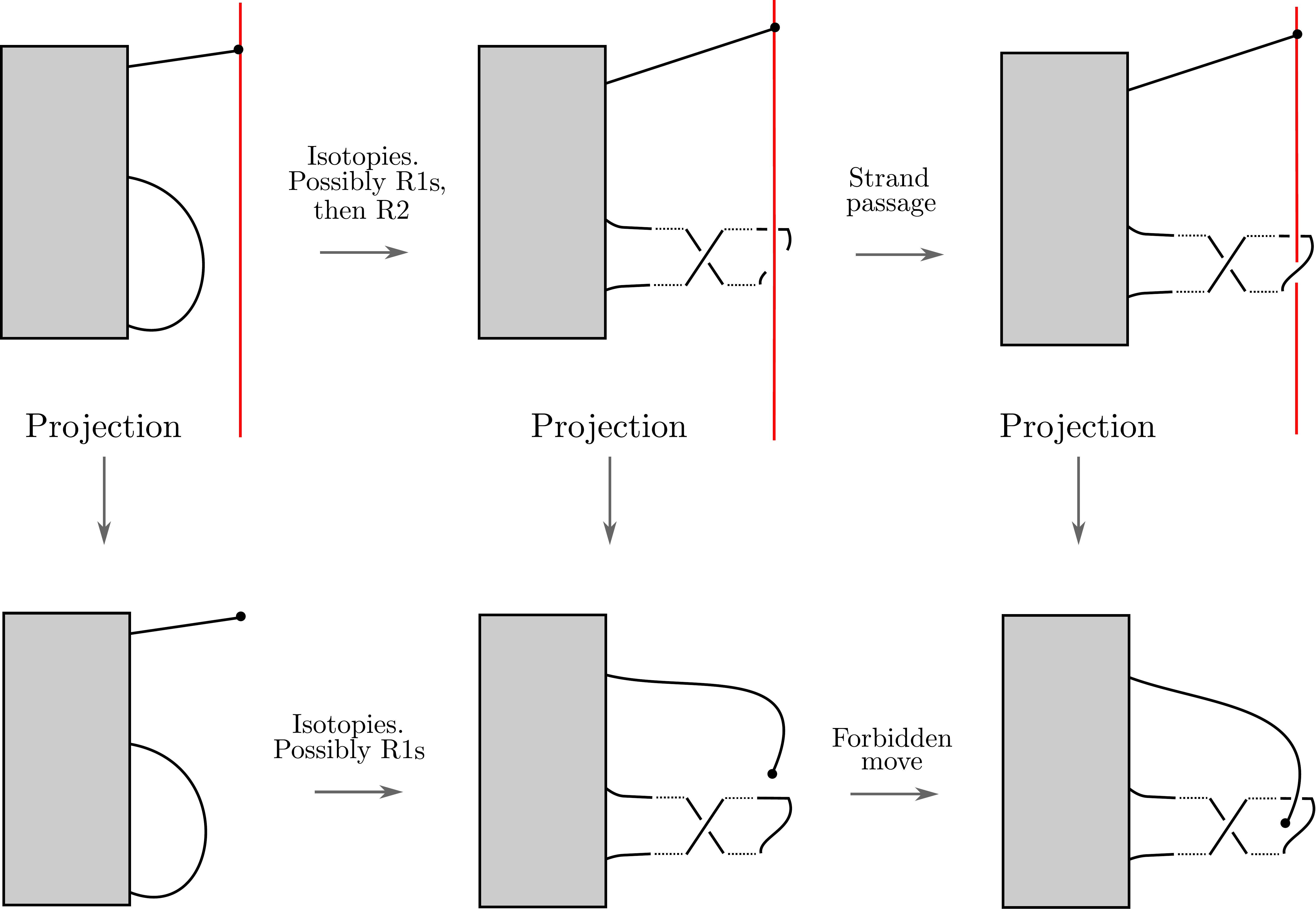}
\caption{On the top row, two $\theta$-curves related by a strand passage between the arc $e_0$ and the arc $e_{\pm}$. Up to label preserving ambient isotopies fixing the circle $e_- \cup e_+$ we can make the strand passage look like in the picture. The effect on the corresponding projections giving the knotoids is to perform a sequence of Reidemeister moves of type I followed by a single forbidden move.} 
\label{fig:localpicture2}
\end{figure}
 
\end{proof}

\subsection{Lower bounds on the $f$-distance}\label{sec:lowerbounds}
We use Theorem \ref{characterisation} to produce lower bounds for the forbidden move-distance between equivalence classes of knotoids up to the four involutions of Figure \ref{fig:symmetries}. With a little abuse of notation, we will still call ``knotoids'' these equivalence classes.\\

The \emph{$H_2$-Gordian distance} $d_{H_2}(K, K')$ between two knots $K$ and $K'$ is defined as the minimal number of equivariant band attachments connecting $K$ and $K'$ (see \cite{kanenobu2010band}). As an immediate consequence of Theorem \ref{characterisation}, given two knotoids $k_1$ and $k_2$, with corresponding strongly invertible knots $\gamma_S(k_1) = (K_1, \tau_1)$ and $\gamma_S(k_2) = (K_2, \tau_2)$, we have that
\begin{equation}\label{bound1}
d_f(k_1, k_2) \geq d_{H_2}(K_1, K_2). 
\end{equation}

Analogously, given two knotoids $k_1$ and $k_2$, consider the pairs $(K^{+}_{k_1}, K^{-}_{k_1})$ and $(K^{+}_{k_2}, K^{-}_{k_2})$. We can define their \emph{Gordian distance} $d_{\text{pair}}((K^{+}_{k_1}, K^{-}_{k_1}), (K^{+}_{k_2}, K^{-}_{k_2}))$ as the minimum between $d( K^{+}_{k_1}, (K^{+}_{k_2} ) + d( K^{-}_{k_1}, (K^{-}_{k_2} ) $ and $d( K^{+}_{k_1}, K^{-}_{k_2} ) + d( K^{-}_{k_1}, K^{+}_{k_2} )$, where $d$ is the usual Gordian distance between knots, that is, the minimal number of crossing changes needed to transform a knot into another. From Remark \ref{rmk:justoneknot} it follows that

\begin{equation}\label{bound2}
 d_f(k_1, k_2) \geq d_{\text{pair}}((K^{+}_{k_1}, K^{-}_{k_1}), (K^{+}_{k_2}, K^{-}_{k_2})).  
\end{equation}

Let $u(k)$ denote the unknotting number of $k$ as a knot in $S^3$, that is, the minimal number of crossing changes needed to unknot $k$ \cite{rolfsen}.
We have as an immediate corollary of Theorem \ref{characterisation} we have the following.

\begin{cor}\label{unknotting}
Let $k_u$ denote the trivial knotoid. If $k$ is a non trivial knot type knotoid then $d_f(k,k_u) \geq 2 u(k)$.
\end{cor}

\begin{proof}
If $k$ is a knot type knotoid, its corresponding constituent knots $K^{\pm}_k$ are both isotopic to $k$. Thus, $d_f(k,u)\geq 2 u(K^{\pm}_k)$ where $u(K^{\pm}_k)$ is the unknotting number of the constituent knot.

\end{proof}

Note that we cannot prove the equality, since in $K^{\pm}_k$ the unknotting crossing change might involve only the arc $e_{\pm}$, and thus it would not correspond to a forbidden move. \\

\begin{rmk}\label{rmk:unknotting}
As for knots and crossing changes, the $f$-distance $d_f(k,0_1)$ between a knotoid $k$ and the trivial knotoid $0_1$ provides a measure of complexity for $k$. We will call this quantity the \emph{unravelling number} of $K$, denoted by $u_f(K)$.

\end{rmk}

\section{Computing $f$-distances of $S^2$-knotoids}\label{sec:computations}

As mentioned above, the main theorem provides lower bounds for $f$-distances between isotopy classes of knotoids. Since our aim is to build a table of $f$-distances, this information alone is not sufficient. For this reason, we computed experimentally the $f$-distance between all non-composite knotoid diagrams, including non-minimal crossing representations, with up to six crossings with the help of a computer program written in \texttt{python 3.7}. 

In brief our method is as follows. First, all  2363766 knotoid diagrams with up to six crossings (both of minimal and non-minimal crossing number representation) \cite{tableknotoids} are encoded using  oriented Gauss codes. We note here that from now on, we shall be using the terms knotoid diagram and oriented Gauss code interchangeably. Each knotoid diagram is then identified using the arrow polynomial for knotoids \cite{gugumcu2017}  and the classification of $S^2$-knotoids found in  \cite{tableknotoids}. Let now $\mathcal{K}$ be the set of all knotoid diagrams with up to six crossings and let $G(V,E)$ be an undirected graph such that:

\begin{align*}
 & V(G) = \mathcal{K} \\
 & E(G) = \{ (v,u) \ \vert \ (v,u) \in \mathcal{K}^2, \ v\stackrel{f}{\sim} u, \ v\neq u\},
 \end{align*} 
where $ v\stackrel{f}{\sim} u$ denotes a pair of knotoid diagrams $(v, u)$ that are related by a single forbidden move. In other words, $G$ is the undirected graph whose vertices are knotoid diagrams and  if two diagrams are related with a single forbidden move, then the corresponding vertices of $G$ are connected with an edge.  Our program builds $G$ and then searches for all Dijkstra paths between all possible pairs of vertices. Finally, the set of all diagrams is partitioned into isotopy classes and the Dijkstra path of minimal distance between two isotopy classes determines their numerical $f$-distance, $d_f^{\rm num}$. From this we can obtain upper bounds for the $f$-distances between isotopy classes of knotoids by computing their experimental $f$-distances which are defined as:

$$d_f^{\rm exp}(v,u) = \min \left \{  d_f^{\rm num}(v,x) \  \vert \ x \in ( u ,u^m,u^s, u^{ms}) \right \}. $$

By comparing the upper bounds with the lower bounds discussed in Section \ref{sec:lowerbounds} we are able to produce Table \textbf{S1} (shown in the Supplementary Information) containing the $f$-distances between equivalence classes of knotoids with minimal crossing number $\leq 4$. Most of the lower bounds in Table \textbf{S1} are obtained using the inequality \ref{bound2}.  Gordian distances between knot types are taken from \cite{moon}, while $H_2$-distances from \cite{kanenobu2010band} and \cite{kanenobu2016band}.

Note that this can possibly be improved by considering in the experimental approach a higher threshold  for the maximum crossing number. This means that non-minimal representations of higher crossing number for the ambiguous entries in Table \textbf{S1} will be considered, which may help decreasing their upper bounds. Unfortunately, our available computational power prohibited us from exploring this possibility.

\begin{ex}
Computing lower bounds using the inequality \ref{bound2} it is quite straightforward. Indeed, given a knotoid $k$ we obtain the constituent knots $K^{\pm}_k$ as explained in Remark \ref{rmk:closure}. Then, using values for the Gordian distance taken from \cite{moon} we compute $d_{\text{pair}}$ for each pair of knotoids.

To illustrate how our method works in the case of inequality \ref{bound1}, we will prove as an example that $d_f(3_1,4_7) = 2$. 
Given a knot $K$, it is well known (see \emph{e.g.}~\cite{rolfsen}) that the double cover of $S^3$ branched along $K$ is a closed $3$-manifold $\Sigma(K)$ whose homeomorphism class depends solely on the knot $K$. Let's denote by $\delta(K)$ the dimension of the first homology of $\Sigma(K)$ with coefficients in $\mathbb{Z}_3$, $\delta(K) = \text{dim}(H_1(\Sigma(K), \mathbb{Z}_3))$. The value of the Jones polynomial of $K$ at $ t^{1/2} = e^{i \pi /6}$ can be computed as $V(K, \omega) = \pm(i \sqrt{3})^{\delta(K)}$ \cite[Proposition $5.1$]{kanenobu2010band}.  If two knots $K$ and $K'$ have $H_2$-distance $1$ then the ratio $ V(K, \omega)/V(K', \omega) \in\{\pm 1, \pm i \sqrt{3}^{\pm 1} \} $\cite[Lemma $5.2$]{kanenobu2010band}. 

We shall apply these to the pair $(3_1,4_7)$. Following \cite{dbc} it is straightforward to see that knotoid $3_1$ lifts to a connected sum of trefoil knots $3_1 \sharp 3_1$. Additionally, it is known (see \emph{e.g.}~\cite{kanenobu2010band}) that $\delta(3_1 \sharp 3_1) = 2$ and, thus, $V(3_1 \sharp 3_1, \omega)  = \pm 3$. On the other hand, $4_7$ lifts to the torus knot $8_{19}$, and it is known (see \emph{e.g.}~\cite{knotorious}) that $H_1(\Sigma(8_{19}), \mathbb{Z}) \cong \mathbb{Z}_3$. Thus, $\delta(8_{19}) = 1$. The ratio $ V(3_1 \sharp 3_1, \omega)/V(8_{19}, \omega) \notin\{\pm 1, \pm i \sqrt{3}^{\pm 1} \} $ and so $3_1$ and $4_7$ cannot have $H_2$-distance equal to 1. Finally, from Table \textbf{S4} (shown in the Supplementary Information) we deduce that $d_f^{\rm exp} (3_1, 4_7) = 2$ and therefore we have that $d_f(3_1,4_7) = 2$. In a similar way we compute lower bounds for the $f$-distance of corresponding to the entries in red of Tables \textbf{S1} and \textbf{S2} (shown in the Supplementary Information), since in these cases the knotoids lift to knots $K$ with $\delta{K} = 1$.

\end{ex}

\section{Application to protein studies}\label{sec:application}

Proteins are long linear biomolecules that often fold into conformations with non-trivial topology. By tracing the coordinates of their ${\rm C}\alpha$ atoms, one can model them as open polygonal curves in 3-space. Until recently, in order to analyse proteins in terms of their knottedness one had first to artificially close  the curve since under classical knot theory all open curves are topologically trivial. In \cite{goundaroulis2017} the second author and collaborators proposed an alternative approach using the concept of knotoids. 

\begin{figure}[!tbh]
\centering
\includegraphics[width=11cm]{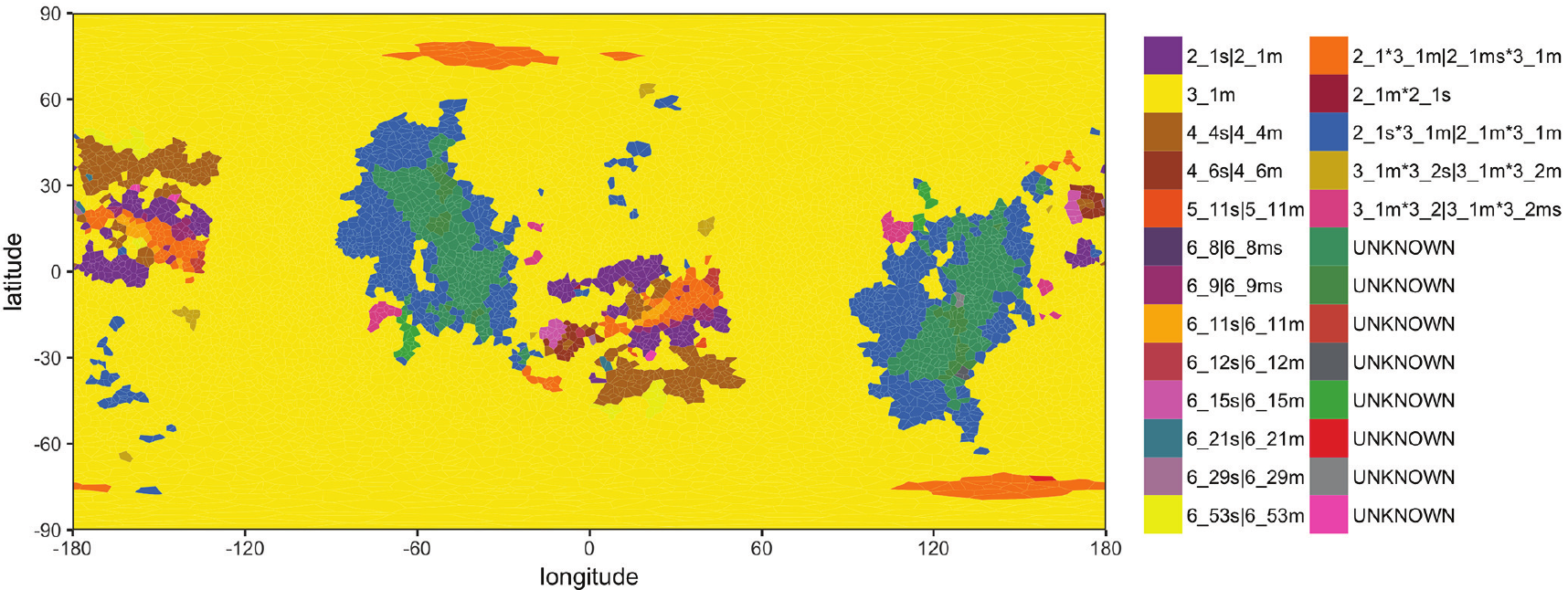}
\caption{The projection map for the protein N-acetyl-L-ornithine transcarbamylase complexed with N-acetyl-L-ornirthine (pdb code: 3kzn). For this map we used 10000 projections. We can see that the predominate knotoid type is $3_1^m$ since it corresponds to the region with the biggest area. The colour scale on the right shows the colour - knotoid type correspondence. Note that the $\vert$ symbol in the knotoid name stands for ``or'', meaning that knotoid names separated with $\vert$ share the same arrow polynomial.  The label Unknown corresponds to knotoids with crossing number $> 6$. Finally, knotoid composition is indicated by an asterisk $\ast$. } 
\label{fig:3kzn}
\end{figure}

The general idea of this approach is to characterise the {\it global topology} of each protein chain by assigning a knotoid type to it. The term global topology refers to the topology of the whole protein chain. The modelled protein is considered inside a large enough sphere (one can also consider its convex hull), centred at the centre of mass of the chain. Each one of the $S^2$-many generic projections on planes around the sphere determines to a knotoid  \cite{goundaroulis2017}. Note that different planes of projection may yield different knotoid diagrams and so determining the knotoid type of a protein using a single projection is far from being accurate.  In principle, the knotoids approach considers the knottedness of any open-ended curve embedded in 3-space as a distribution (also called {\it spectrum}) of knotoid types. In what follows we will call spectrum the list of different knotoid types appearing in the distribution. The knotoid with the highest probability in the distribution of knotoid types over all projections, characterises the protein and is called {\it predominate}. In order to obtain an unbiased overview of a protein's topology, all possible projections have to be considered but since this is not computationally feasible, the distribution is approximated by sampling from the space of all possible projections. 
To avoid a change of knotoid type under ambient isotopy, two infinite lines are introduced each time a projection plane is chosen. Each line passes through one of the endpoints and they are perpendicular to the projection plane \cite{gugumcu2017, goundaroulis2017}. Additionally, an algorithm similar to KMT \cite{muthukumar,taylor} that simplifies the curve but preserves its underlying topology is also applied \cite{gugumcu2017, goundaroulis2017} in order to make computations of knotoid types more efficient.
The knotoid types are determined using invariants. 
For this work we have used the arrow polynomial for knotoids in $S^2$ \cite{gugumcu2017}.

The above are often summarised in a plot called {\it the projection map} \cite{goundaroulis2017}. The projection map is in fact the Voronoi diagram $\mathcal{VD}$ of the corresponding Delaunay triangulation of $S^2$ with respect to the set of points sampled from $S^2$. Furthermore, each cell of the projection map is colour-coded according to the knotoid type it produces. By construction, there is a bijection between the number of different colours in a projection map and the number of different knotoids in the spectrum of the analysed curve (see Fig.~\ref{fig:3kzn}). 
Both the spectrum and the projection map depend heavily on the sample size of projections; if too few points are sampled, then the overall topology of the analysed curve will not be well approximated. In fact, the optimal size for the set of sampled projections remains an open question.

\begin{figure}[!tbh]
\includegraphics[scale=1]{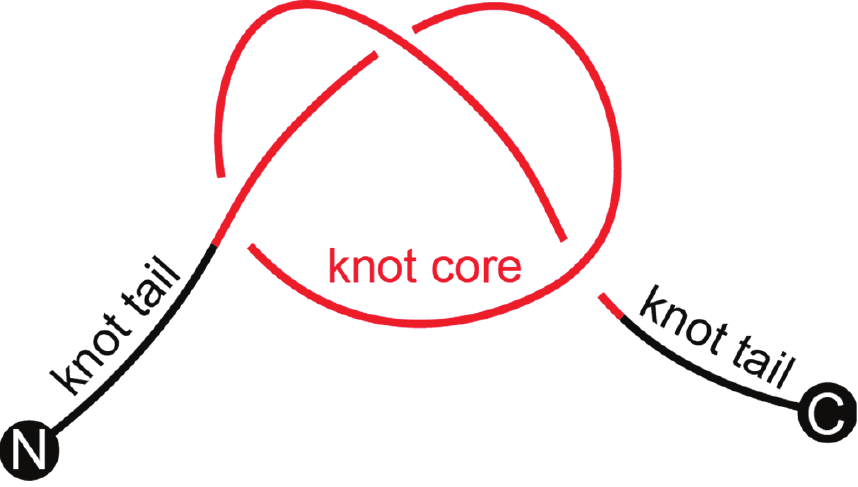}
\caption{The core (in red) and the two tails (in black) of an open knot representing a protein with a trefoil knot. The two beads on either side represent the two termini, N and C respectively, of the protein chain.} 
\label{fig:knot_core}
\end{figure}

In several cases the protein chain doubles back right after forming a knot. This results in unknotting the knot previously formed. This type of proteins are called slipknotted. In order to detect these local knots one has to study the {\it local topology} of any given protein by analysing all possible subchains. During this analysis one can also determine the {\it knotted core} of a protein, that is the shortest subchain forming a knot. The subchains located before and after the knotted core are called {\it N-tail} and the {\it C-tail}, where N and C are the two termini of the protein (see Fig.~\ref{fig:knot_core}). The vast majority of proteins are enzymes where there is an overlap between knotted cores and the respective enzymatic active sites. Indeed, these sites are either located inside or close to the knotted core of the chain. In this context, it was shown that knotted cores of proteins play a vital role in some aspects of a protein's structure and function \cite{mallam2007, christian}. Moreover, in \cite{dabrowski2016}, it was observed that formation of deep knots with characteristic structural motifs provides a favourable environment for active sites in enzymes. However, it is important to mention that the knots are not necessary for the formation of regions with increased intra-chain contacts.  The subchain analysis can be computationally heavy, depending on the total length of the protein. Therefore, it would be useful if one could probe the depth of a knot straight from the global topology analysis.

In this section we will use Theorem~\ref{characterisation} to provide approximations for the two questions that were discussed in this section, namely:
\begin{enumerate}
\item[a.] How many projections are required in order to have an accurate overview of a protein's topology?
\item[b.] Can we determine whether a protein is deeply knotted or not by looking at its knotoid distribution?
\end{enumerate}

\subsection{Approximating the sample size of projections}\label{sec:application1}
Consider a generic projection  of a protein chain on some plane and let $k$ be the corresponding knotoid. If we continuously perturb the projection direction until the knotoid type changes to $k^\prime$, we will obtain a pair of knotoids with $d_f (k ,k^\prime) =1$. 
We will use this idea to make measurements on projection maps that are obtained from sample sets of increasing size, in order to approximate numerically an optimal sample size  of projections $s$ for a given protein.

As mentioned earlier, the spectrum of a protein chain depends on the number of projections. Therefore, there is a higher chance for cells corresponding to knotoids with $d_f >1$ to appear next to each other. Since the predominate knotoid corresponds to the largest region of the projection map, it is suffices to focus on the discrepancies between the region of the predominate and its immediate neighbours. For this, we define the \emph{interface error}, $er(s)$, associated to the sample of size $s$.  The interface error is the ratio of the total number of  pairs of adjacent regions adjacent to the region of the predominate knotoid $k_0$ that correspond to knotoids $k_i$, for which Theorem~\ref{characterisation} gives $d_f(k_0 ,k_i)>1$ over the number of all adjacent regions to $k_0$. By gradually increasing the number of projections, the triangulation will become progressively finer and so the possibility of having pairs of adjacent cells with $d_f>1$ will effectively decrease, hence the error $er(s)$ will decrease.

For our experiment, we concentrated on proteins with predominate knotoid type $3_1$. There are $517$\footnote{As for August 2020.} such proteins in total deposited in the Protein Data Bank \cite{pdb}, according to  \cite{knotprot1, knotprot}. All proteins of interest are analysed using 50, 100, 500, 1000, 5000 and, finally, 10000 projections.  Each time $er(s)$ for the respective Voronoi diagram is computed. In more detail, for each Voronoi diagram we build a graph  where the vertices correspond to the regions of the map and the edges correspond to common boundaries between regions. We compute $er(s)$ by counting the number of graph edges between $3_1$ and knotoids that give $d_f>1$ and taking the ratio over the total number of edges that have $3_1$ as one of its endpoints, loops excluded. In Table \textbf{S2} (in the Supplementary Information) we present the $f$-distances of $3_1$ from all knotoids with 5 and with 6 crossings.  The number of unique knotoid types in the Voronoi diagram corresponds to the size of $Spec(s)$. Note that in this analysis we don't consider composite knotoids. Finally, we compute the average interface error for one each of the six cases of projections sample size. 

\begin{figure}[!tbh]
\includegraphics[scale=.6]{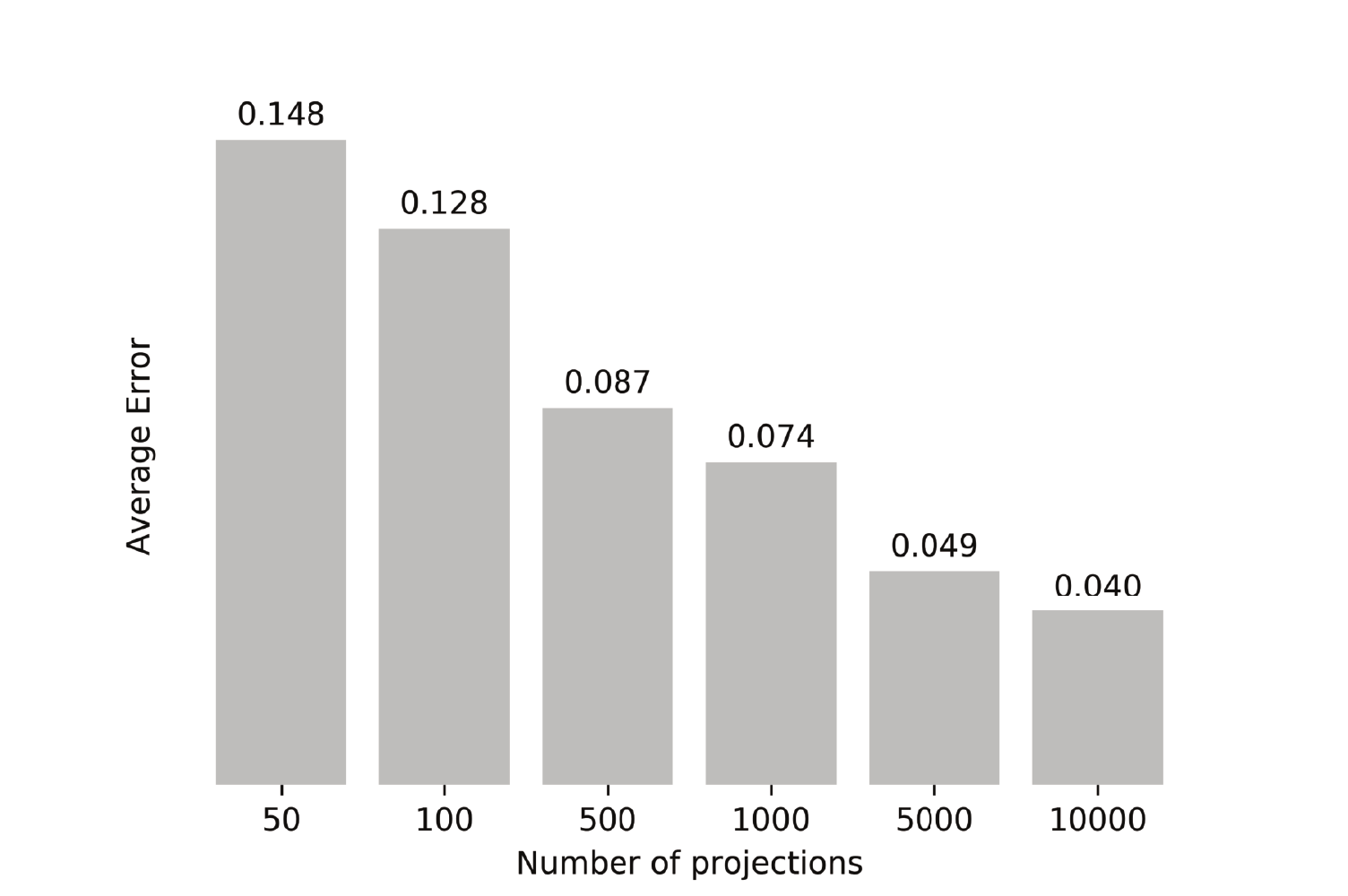}
\caption{The diminishing effect of increasing the number of projections on the average interface error $er(s)$ of a protein chain.} 
\label{fig:correlations}
\end{figure}

As shown in Fig.~\ref{fig:correlations}, we observe a gradual decrease in the average error, as the number of projections is increased. Moreover, the difference between 50 and 100 projections is somehow statistically significant (p = 0.010, Mann-Whitney U test) while the differences for groups with greater or equal than 100 projections is always statistically significant (p-value negligible, Mann-Whitney U test).

Our analysis suggests that at 10000 projections we will have the most accurate overview of the topology of a protein. 

This is a number of projections that can be handled very efficiently by several online servers for knotted proteins, e.g. \cite{knotprot}. However, our aim is to analyse accurately a protein locally using a personal computer. With this in mind, we assume that 5000 projections provide an approximation reliable enough.

\subsection{Probing the depth of knotted proteins}\label{sec:application2}
In this section we discuss a method to estimate if a protein is deeply knotted, without passing through the computational expensive subchain analysis \cite{knotprot}. 

\begin{defi}
Let $k$ be a knotoid in $S^2$. A subknotoid of $k$, denoted by $k_{\rm s}$, is a knotoid such that:
\[
d_f(k, k_{\rm s}) < u_f( k) \ \  \mbox{and} \ \ u_f(k_{\rm s}) < u_f(k)
\]
Where $u_f(k_{\rm s})$ and $u_f(k)$ are the unravelling numbers of $k_{\rm s}$  and $k$ (recall Remark \ref{rmk:unknotting}).
\end{defi}

Unknotting $k$ via forbidden moves induces  a sequence of non-trivial subknotoids.  We denote this set of knotoids by $\mathcal{U}_{k}$, the unravelling set of $k$.\\

A deep protein knot (or more general an open polygonal knot) is also usually tightly knotted as the length of its knotted core is relatively small, compared to the knot's overall length. The converse is not always true  since we can easily find an example of a tight knot very close to one the termini of the chain. Our assumption is that the deeper a knot is, the smaller the total area of the subknotoids will be in the knotoid distribution. This is because  the two tails of the knot are less probable to interact with the rest of the chain in a way that will produce an unknotting forbidden move.  Similarly, we argue that this will increase the probability of having regions corresponding to more complex knotoids. To quantify this assumption we define the notions of  \emph{Unravelling Area}  $\mathcal{A}_{\text{un}}$ and \emph{Weighted Area}  $\mathcal{A}_{\text{w}}$ as two different sums, namely:

\begin{equation}\label{eq:areaI}
\mathcal{A}_{\rm un} =  1 - \frac{\sum_{k \in \mathcal{U}_{k^p}} \mathcal{A}_k}{\mathcal{A}}
\end{equation}

\begin{equation}\label{eq:areaW}
\mathcal{A}_{\rm w} =  \frac{\sum_{k \in \mathcal{S}} u_f(k) \cdot  \mathcal{A}_k}{u_f(k^p) \cdot \mathcal{A}},
\end{equation}

where $k^p$ is the predominate knotoid, $\mathcal{U}_{k^p}$ is the unravelling  set of $k^p$, $\mathcal{S}$ is the knotoid spectrum, $\mathcal{A}$ is the area of the sphere, and $\mathcal{A}_k$ is the total area in the sphere of regions that correspond to knotoid $k$.

\begin{rmk}
Note that by construction $\mathcal{A}_{\rm un}$ take values between $0$ and $1$, and it defines a meauser of how ``simple'' the protein is. The closer $\mathcal{A}_{\rm un}$ is to $0$, the ``simpler'' our protein is with respect to a curve having only the predominate $k^p$  in its knotoid spectrum (or $k^p$ and knotoids with $u_f(k) \geq u_f(k^p)$). Similarly, $\mathcal{A}_{\rm w}$ can be seen as a measure of how much the geometry and topology of a protein differ from the case of a curve having only $k^p$ in its knotoid spectrum, with $\mathcal{A}_{\rm w}>1$ corresponding to more complex geometries. 
\end{rmk}

As mentioned earlier, the depth of a protein knot is determined by progressively trimming the chain from each side and evaluating its knot type until the knotted core is obtained. Moreover, as mentioned in the introduction, proteins are often considered deep when at least 20 amino acids are trimmed from either of their termini before converting their knot type. This definition however can be problematic for large proteins with large and loose knotted cores. 

For this reason we consider $D$, the depth of a knotted (open) curve as:
\[
D(p) = \frac{\ell_N (p)\, \ell_C (p)}{\ell_T(p)^2}
\] 
where $p$ is the open curve representing the protein chain, $\ell_T(p)$ is the total length of the chain $p$, $\ell_N(p)$ is the length of the N-tail of $p$, and $\ell_C(p)$ is the length of the C-tail of $p$. By definition, $D$ is continuous and reflects the fact that some proteins are more deeply knotted than others.

\begin{figure}[!tbh]
\includegraphics[scale=.6]{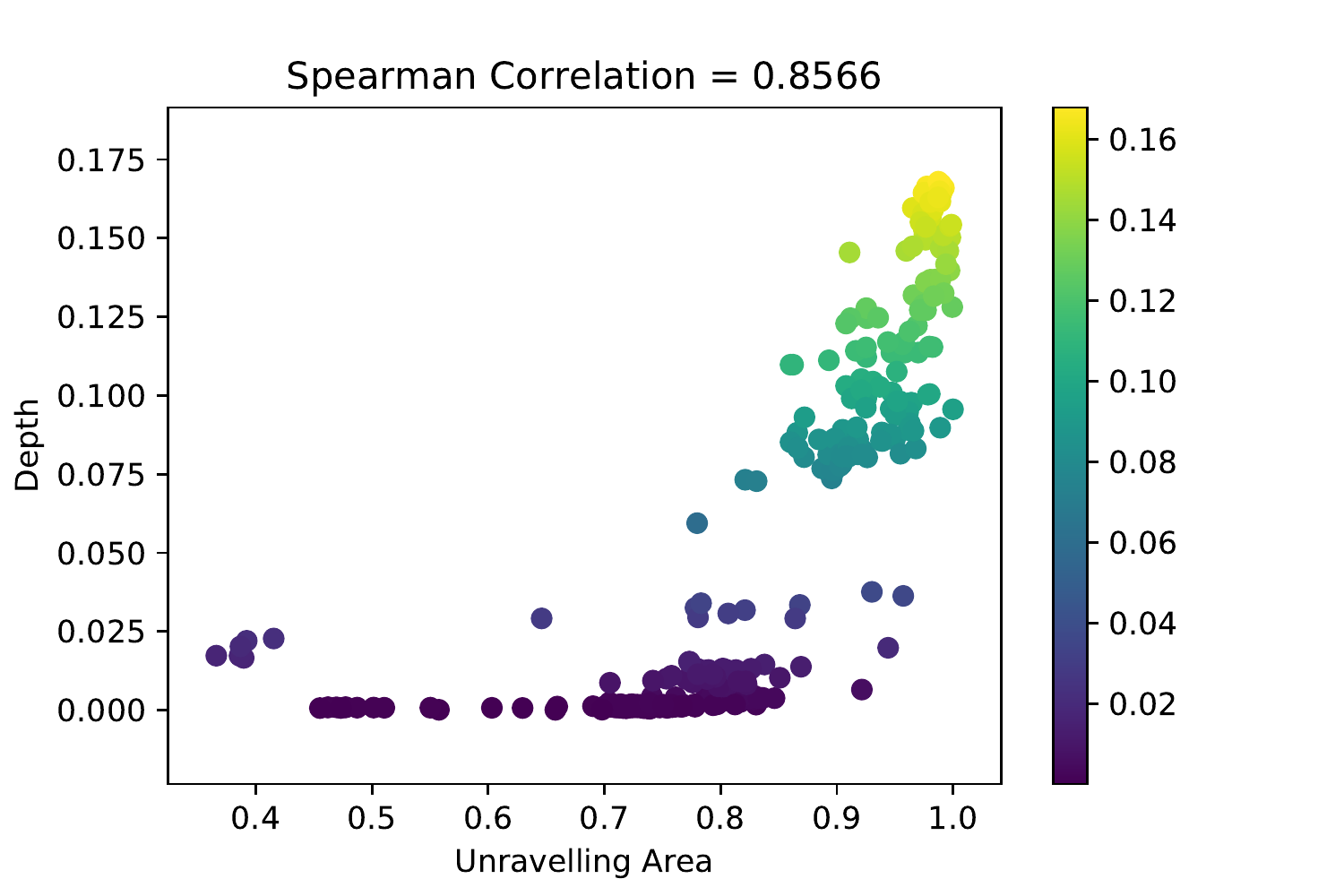}
\caption{Scatterplot of values of $\mathcal{A}_{\rm un}$ against $D(p)$. The color map on the right indicates the different values of $D(p)$. The higher the value, the deeper a knot is. Two distinct clusters of points, in terms of $D(p)$, are visible in the graph indicating a well defined separation between deeply and shallowly knotted proteins.} 
\label{fig:scatter}
\end{figure}

In the spirit of the previous section, we will work with all proteins that form a $3_1$ knot.  Since from Theorem \ref{characterisation} we have that $u_f(3_1) =2$,  the subknotoids for $3_1$ are such that $d_f(3_1, k_{s}) < 2 \ \  \mbox{and} \ \ u_f(k_{s}) < 2$. Theorem \ref{characterisation} again implies that $2_1$ is the only knotoid (among knotoids with $6$ or fewer crossings) having distance $1$ from both $3_1$ and $0_1$. In fact, it is straightforward to check that all the knotoids $k$ with less than $6$ crossings and $d_f(k,3_1) =1$ have distance $>1$ from $0_1$ using inequality \ref{bound2}. In this case Eq.~\ref{eq:areaI} becomes: 

\[ 
\mathcal{A}_{\rm un} =  \frac{\mathcal{A}_{0_1}+ \mathcal{A}_{2_1}}{\mathcal{A}}
\]

Next, we plotted the pairs \large($\mathcal{A}_{\rm un}$, $D(p)$\large) and \large($\mathcal{A}_{\rm w}$, $D(p)$\large) for all of the 517 studied proteins (see Figs~\ref{fig:scatter} and \ref{fig:scatter2}). In more detail, we first compute the projection map for each protein using the optimal value of 5000 projections that was determined in the previous section. From each projection map we then compute the corresponding $\mathcal{A}_{\rm un}$, $\mathcal{A}_{\rm w}$.  The quantities $\ell_N$, $\ell_C$ and $\ell_T$ that are required to compute $D(p)$ are taken from \cite{knotprot}. The unravelling numbers of the various knotoids can be found in Table \textbf{S3} (shown in the Supplementary Information), and are obtained using inequalities \ref{bound2} and \ref{bound1}.

\begin{figure}[!tbh]
\includegraphics[scale=.6]{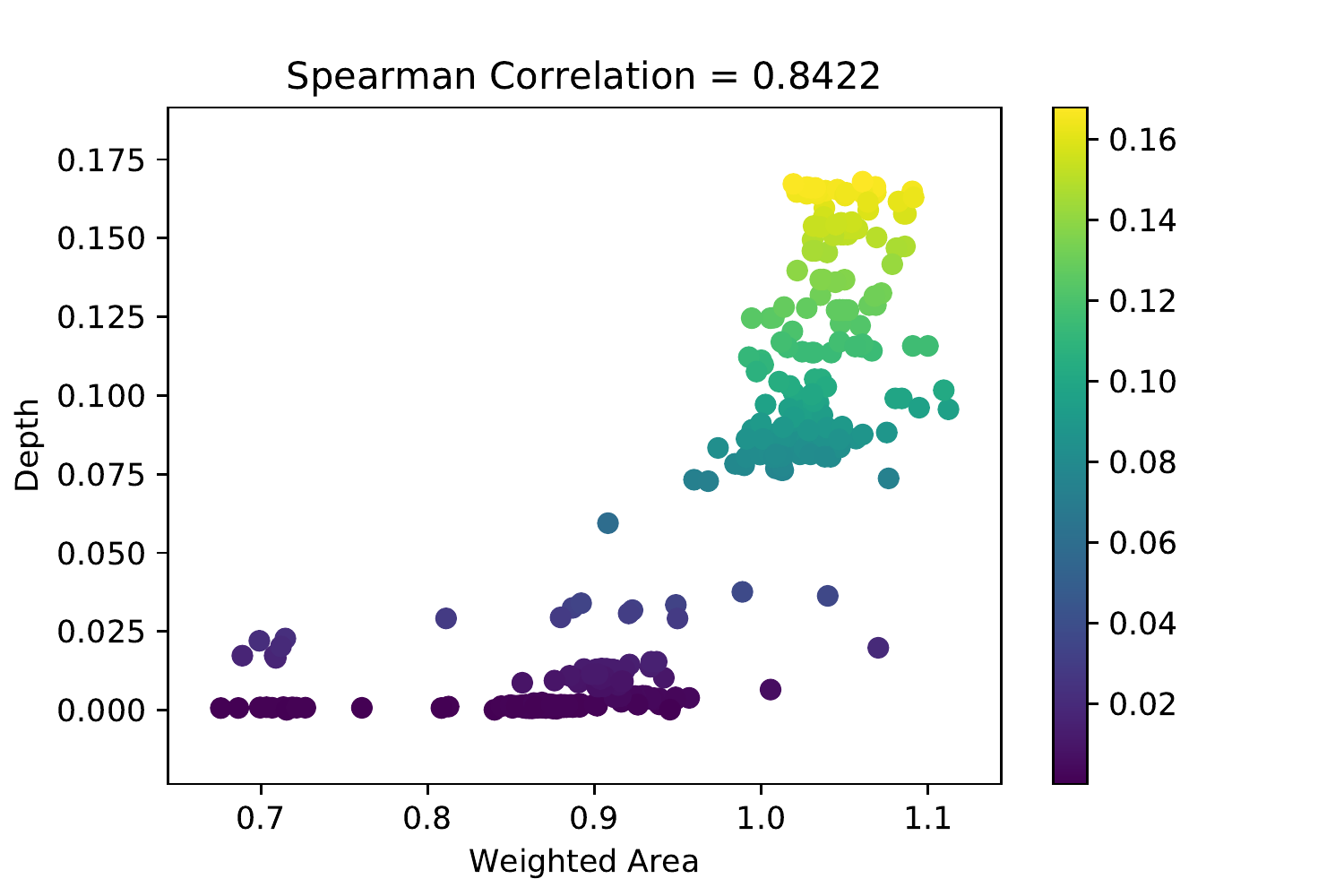}
\caption{Scatterplot of values of $\mathcal{A}_{\rm w}$ against $D(p)$. The color map on the right indicates the different values of $D(p)$. The higher the value, the deeper a knot is. Two distinct clusters of points, in terms of $D(p)$, are visible in the graph indicating a well defined separation between deeply and shallowly knotted proteins.} 
\label{fig:scatter2}
\end{figure}

The Spearman correlation between $\mathcal{A}_{\rm un}$ and $D(p)$ is 0.8566, while between $\mathcal{A}_{\rm w}$ and $D(p)$ is 0.8422, indicating a strong monotonous relation in both cases. The key observation here is that in both figures, the proteins form  two separate clusters, one in upper right corner of the scatterplot and one in the lower-lower left. In the lower cluster includes $3_1$-proteins with lower values of $D$ and mainly low-to-medium values of $\mathcal{A}_{\rm un}$ and $\mathcal{A}_{\rm w}$ respectively, while the upper one contains $3_1$-proteins with higher values. This clustering suggests a value of $D(p) = 0.06$ as threshold to classify as deeply knotted or shallow a given protein.  We note here that there are four proteins (PDB codes: 5yud, 1by7, 4h6v and 1f48) that despite having relatively small values of $D$, they have quite complex knotoid distributions. This could be due to the fact that in all cases the knotted core is located very close to one of the two termini and thus only a few number of amino acids need to be trimmed. This leaves a rather long tail on the other side that can potentially interact with the rest of the chain, forming more complex knotoids upon projection.

Finally, the histograms in Figure \ref{fig:histogram} of the two different groups indicate that, in principle, a value $\mathcal{A}_{\rm un} > 0.85$ and a value of $\mathcal{A}_{\rm w} > 0.95$ most probably suggests that a protein has deeply knotted trefoil.

\begin{figure}[ht]
\begin{subfigure}[t!]{.8\textwidth}
  \centering
  \includegraphics[width=.8\linewidth]{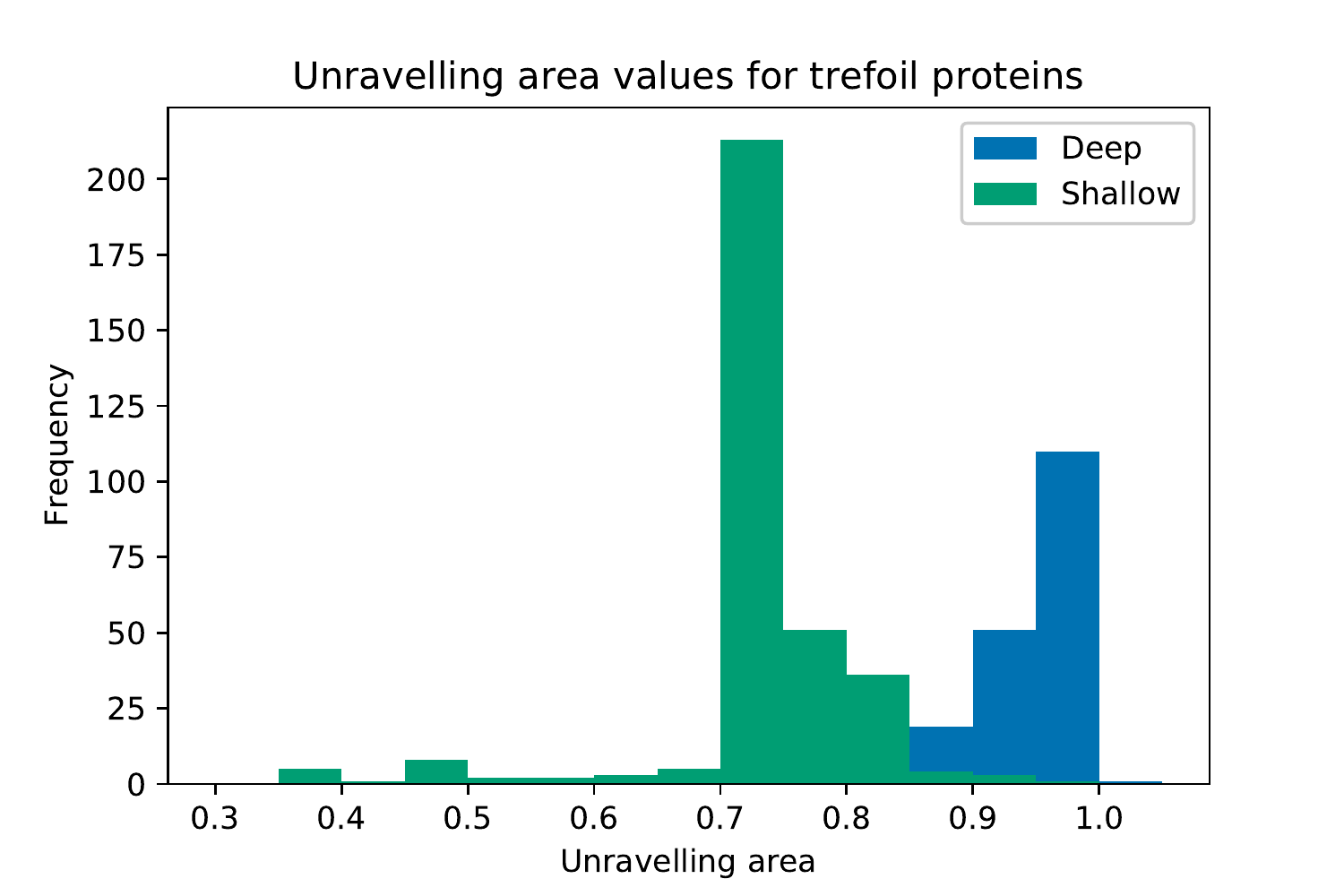}  
\end{subfigure}
\begin{subfigure}[t!]{.8\textwidth}
  \centering
  \includegraphics[width=.8\linewidth]{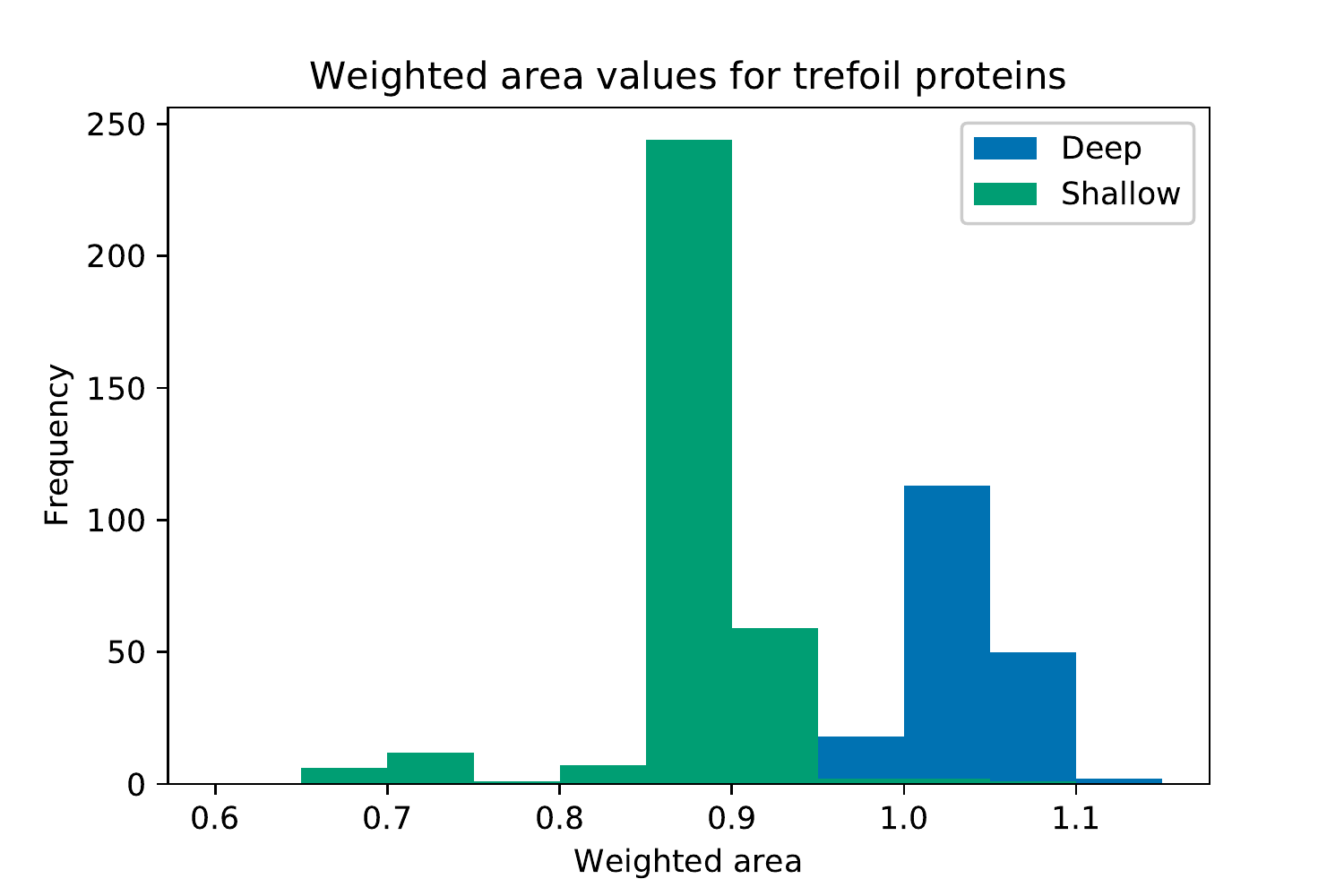}  
\end{subfigure}
\caption{The two histograms for the deeply knotted protein group and the shallowly knotted group. Here we consider deep those proteins $p$ with $D(p) > 0.06$. }
\label{fig:histogram}
\end{figure}

Concluding, our computations show how, remarkably, we can infer subtle information about the geometry of the protein and about its knot depth, directly from a refined  topological analysis based on the properties of knotoids and on Theorem \ref{characterisation}.  It would be interesting to further explore $\mathcal{A}_{\rm un}$ and $\mathcal{A}_{\rm w}$ for other proteins knots and determine the appropriate cutoff values that suggest deep knots of specific type.

\section*{Acknowledgements}
The authors would like to thank Dorothy Buck, Heather Harrington, Marc Lackenby and Andrzej Stasiak for several fruitful conversations and guidance throughout the development of this work. A.B. acknowledges support by the RS-EPSRC grant ``Algebraic and topological approaches for genomic data in molecular biology'' EP/R005125/1. A.B. is part of the Oxford TDA centre and is currently a Hooke research fellow.
We also thank the COST Action European Topology Interdisciplinary Action (EUTOPIA) CA17139 for supporting collaborative meeting of the authors.

\vspace{0.8cm}
\textbf{Author contributions:}
AB and DG conceived, designed and performed the research. AB and DG wrote, revised and approved the manuscript. Both authors contributed equally to this work.
\vspace{0.8cm}

\textbf{Data availability:} The data is available on the GitHub repository: https://github.com/dgound/f-distance

\newpage  

\bibliography{knotoidsdistance1}{}

\begin{thebibliography}{10}

\bibitem{dbc}
Agnese Barbensi, Dorothy Buck, Heather~A Harrington, and Marc Lackenby,
  \emph{Double branched covers of knotoids}, arXiv preprint arXiv:1811.09121,
  to appear in Communications in Analisys and Geometry (2018).

\bibitem{pdb}
Helen~M Berman, Philip~E Bourne, John Westbrook, and Christine Zardecki,
  \emph{The protein data bank},  (2003), 394--410.

\bibitem{knotorious}
M~Borodzik and S~Friedl, \emph{Knotorious world wide web page}, 2011.

\bibitem{christian}
Thomas Christian, Reiko Sakaguchi, Agata~P Perlinska, Georges Lahoud, Takuhiro
  Ito, Erika~A Taylor, Shigeyuki Yokoyama, Joanna~I Sulkowska, and Ya-Ming Hou,
  \emph{Methyl transfer by substrate signaling from a knotted protein fold},
  Nature structural \& molecular biology \textbf{23} (2016), no.~10, 941.

\bibitem{dabrowski2015prediction}
P~Dabrowski-Tumanski, AI~Jarmolinska, and JI~Sulkowska, \emph{Prediction of the
  optimal set of contacts to fold the smallest knotted protein}, Journal of
  Physics: Condensed Matter \textbf{27} (2015), no.~35, 354109.

\bibitem{knotprot}
Pawel Dabrowski-Tumanski, Pawel Rubach, Dimos Goundaroulis, Julien Dorier,
  Piotr Su{\l}kowski, Kenneth~C Millett, Eric~J Rawdon, Andrzej Stasiak, and
  Joanna~I Sulkowska, \emph{Knotprot 2.0: a database of proteins with knots and
  other entangled structures}, Nucleic acids research \textbf{47} (2018),
  no.~D1, D367--D375.

\bibitem{dabrowski2016}
Pawel Dabrowski-Tumanski, Andrzej Stasiak, and Joanna~I Sulkowska, \emph{In
  search of functional advantages of knots in proteins}, PloS one \textbf{11}
  (2016), no.~11, e0165986.

\bibitem{knotoID}
Julien Dorier, Dimos Goundaroulis, Fabrizio Benedetti, and Andrzej Stasiak,
  \emph{Knoto-id: a tool to study the entanglement of open protein chains using
  the concept of knotoids}, Bioinformatics \textbf{34} (2018), no.~19,
  3402--3404.

\bibitem{goundaroulis2017}
Dimos Goundaroulis, Julien Dorier, Fabrizio Benedetti, and Andrzej Stasiak,
  \emph{Studies of global and local entanglements of individual protein chains
  using the concept of knotoids}, Scientific reports \textbf{7} (2017), no.~1,
  6309.

\bibitem{tableknotoids}
Dimos Goundaroulis, Julien Dorier, and Andrzej Stasiak, \emph{A systematic
  classification of knotoids on the plane and on the sphere}, arXiv preprint
  arXiv:1902.07277 (2019).

\bibitem{goundaroulis2017topological}
Dimos Goundaroulis, Neslihan G{\"u}g{\"u}mc{\"u}, Sofia Lambropoulou, Julien
  Dorier, Andrzej Stasiak, and Louis Kauffman, \emph{Topological models for
  open-knotted protein chains using the concepts of knotoids and bonded
  knotoids}, Polymers \textbf{9} (2017), no.~9, 444.

\bibitem{gugumcu2017}
Neslihan G{\"u}g{\"u}mc{\"u} and Louis~H Kauffman, \emph{New invariants of
  knotoids}, European Journal of Combinatorics \textbf{65} (2017), 186--229.

\bibitem{band2}
Jim Hoste, Yasutaka Nakanishi, and Kouki Taniyama, \emph{Unknotting operations
  involving trivial tangles}, Osaka Journal of Mathematics \textbf{27} (1990),
  no.~3, 555--566.

\bibitem{knotprot1}
Michal Jamroz, Wanda Niemyska, Eric~J Rawdon, Andrzej Stasiak, Kenneth~C
  Millett, Piotr Su{\l}kowski, and Joanna~I Sulkowska, \emph{Knotprot: a
  database of proteins with knots and slipknots}, Nucleic acids research
  \textbf{43} (2014), no.~D1, D306--D314.

\bibitem{kanenobu2010band}
Taizo Kanenobu, \emph{Band surgery on knots and links}, Journal of Knot Theory
  and Its Ramifications \textbf{19} (2010), no.~12, 1535--1547.

\bibitem{kanenobu2016band}
\bysame, \emph{Band surgery on knots and links, iii}, Journal of Knot Theory
  and Its Ramifications \textbf{25} (2016), no.~10, 1650056.

\bibitem{muthukumar}
Kleanthes Koniaris and Murugappan Muthukumar, \emph{Self-entanglement in ring
  polymers}, J. Chem. Phys. \textbf{95} (1991), no.~4, 2873--2881.

\bibitem{mallam2007}
Anna~L Mallam and Sophie~E Jackson, \emph{The dimerization of an
  $\alpha$/$\beta$-knotted protein is essential for structure and function},
  Structure \textbf{15} (2007), no.~1, 111--122.

\bibitem{mallam2012knot}
\bysame, \emph{Knot formation in newly translated proteins is spontaneous and
  accelerated by chaperonins}, Nature chemical biology \textbf{8} (2012),
  no.~2, 147.

\bibitem{moon}
Hyeyoung Moon, \emph{Calculating knot distances and solving tangle equations
  involving montesinos links},  (2010).

\bibitem{rolfsen}
Dale Rolfsen, \emph{Knots and links}, vol. 346, American Mathematical Soc.,
  2003.

\bibitem{san2017knots}
{\'A}lvaro San~Mart{\'\i}n, Piere Rodriguez-Aliaga, Jos{\'e}~Alejandro Molina,
  Andreas Martin, Carlos Bustamante, and Mauricio Baez, \emph{Knots can impair
  protein degradation by atp-dependent proteases}, Proceedings of the National
  Academy of Sciences \textbf{114} (2017), no.~37, 9864--9869.

\bibitem{soler2013effects}
Miguel~A Soler and Patr{\'\i}cia~FN Fa{\'\i}sca, \emph{Effects of knots on
  protein folding properties}, PloS one \textbf{8} (2013), no.~9, e74755.

\bibitem{soler2014effects}
Miguel~A Soler, Ana Nunes, and Patr{\'\i}cia~FN Fa{\'\i}sca, \emph{Effects of
  knot type in the folding of topologically complex lattice proteins}, The
  Journal of chemical physics \textbf{141} (2014), no.~2, 07B607\_1.

\bibitem{sulkowska2012energy}
Joanna~I Su{\l}kowska, Jeffrey~K Noel, and Jose~N Onuchic, \emph{Energy
  landscape of knotted protein folding}, Proceedings of the National Academy of
  Sciences \textbf{109} (2012), no.~44, 17783--17788.

\bibitem{sulkowska2012conservation}
Joanna~I Su{\l}kowska, Eric~J Rawdon, Kenneth~C Millett, Jose~N Onuchic, and
  Andrzej Stasiak, \emph{Conservation of complex knotting and slipknotting
  patterns in proteins}, Proceedings of the National Academy of Sciences
  \textbf{109} (2012), no.~26, E1715--E1723.

\bibitem{taylor}
William~R. Taylor, \emph{A deeply knotted protein and how it might fold},
  Nature \textbf{406} (2000), 916--919.

\bibitem{taylor2003protein}
William~R Taylor and Kuang Lin, \emph{Protein knots: a tangled problem}, Nature
  \textbf{421} (2003), no.~6918, 25--25.

\bibitem{turaev}
Vladimir Turaev et~al., \emph{Knotoids}, Osaka Journal of Mathematics
  \textbf{49} (2012), no.~1, 195--223.

\bibitem{wald}
Friedhelm Waldhausen, \emph{{\"U}ber involutionen der 3-sph{\"a}re}, Topology
  \textbf{8} (1969), no.~1.

\end{thebibliography}
\bibliographystyle{amsplain}

\providecommand{\bysame}{\leavevmode\hbox to3em{\hrulefill}\thinspace}
\providecommand{\MR}{\relax\ifhmode\unskip\space\fi MR }
\providecommand{\MRhref}[2]{%
  \href{http://www.ams.org/mathscinet-getitem?mr=#1}{#2}
}
\providecommand{\href}[2]{#2}

\newpage

\section*{Supplementary Information}
\subsection*{The $f$-distances tables}\label{appendix:a}

We recall here the inequalities giving the lower bounds for the $f$-distance between two knotoids $k_1$ and $k_2$, in terms of the $H_2$-distance of their lifts $\gamma_S(k_1) = (K_1, \tau_1)$ and $\gamma_S(k_2) = (K_2, \tau_2)$ and in terms of the \emph{Gordian distance} of the corresponding pairs $(K^{+}_{k_1}, K^{-}_{k_1})$ and $(K^{+}_{k_2}, K^{-}_{k_2})$ (see the Main Manuscript for all the details).

\begin{equation}\label{bound1}
d_f(k_1, k_2) \geq d_{H_2}(K_1, K_2). 
\end{equation}

\begin{equation}\label{bound2}
 d_f(k_1, k_2) \geq d_{\text{pair}}((K^{+}_{k_1}, K^{-}_{k_1}), (K^{+}_{k_2}, K^{-}_{k_2})).  
\end{equation}

\vspace{1cm}
Below, we show Table \ref{tab:cr4}, containing he $f$-distances between knotoids (up to rotation, symmetric reflection and mirror reflection) with minimal crossing number $\leq 4$, and Table \ref{tab:cr6}, containing the $f$-distances between the $3_1$ knotoid and knotoids (up to rotation, symmetric reflection and mirror reflection) with minimal crossing number $\leq 6$.

\footnotesize

\begin{table}[H]
%\rowcolors{2}{gray!25}{white}
\centering
\begin{tabular}{@{}lllllllllllll@{}}
\toprule
          & $0_1$   & $2_1$   & $3_1$   & $3_2$     & $4_1$     & $4_2$     & $4_3$     & $4_4$     & $4_5$      & $4_6$      & $4_7$      & $4_8$     \\ \midrule
$0_1$     & 0       & 1       & \blue{2} & 1         & \blue{2}   & 1         & \blue{2}   & \blue{2}   & \blue{2}    & \blue{3}    & 1          & 1         \\
$2_1 $    & 1       & 0       & 1       & \blue{2}   & \blue{3}   & 1         & 1         & 1         & 1          & \blue{2}    & 1          & 1         \\
$3_1  $   & \blue{2} & 1       & 0       & \blue{3}   & \blue{4}   & \blue{2}   & \blue{2}   & 1         & \blue{2}    & 1          & \red{2}     & \blue{2}   \\
$3_2   $  & 1       & \blue{2} & \blue{3} & 0         & 1         & \blue{2}   & \blue{2-3} & \blue{3}   & 1          & \blue{3-4}  & \blue{2}    & \blue{2}   \\
$4_1    $ & \blue{2} & \blue{3} & \blue{4} & 1         & 0         & \blue{3}   & \blue{3-4} & \blue{4}   & \blue{2}    & \blue{ 4-5} & \blue{3}    & \blue{3}   \\
$4_2 $    & 1       & 1       & \blue{2} & \blue{2}   & \blue{3}   & 0         & \grey{1-2} & 1         & \blue{2}    & \blue{2}    & \grey{ 1-2} & 1         \\
$4_3 $    & \blue{2} & 1       & \blue{2} & \blue{2-3} & \blue{3-4} & \grey{1-2} & 0         & \blue{2}   & \blue{2}    & 1          & \grey{1-2}  & 1         \\
$4_4 $    & \blue{2} & 1       & 1       & \blue{3}   & \blue{ 4}  & 1         & \blue{ 2 } & 0         & \blue{2}    & 1 & \grey{1-2}  & 1         \\
$4_5 $    & \blue{2} & 1       & \blue{2} & 1         & \blue{2}   & \blue{2}   & \blue{2}   & \blue{2}   & 0          & \blue{2-3}  & \grey{1-2}  & \blue{2}   \\
$4_6 $    & \blue{3} & \blue{2} & 1       & \blue{3-4} & \blue{4-5} & \blue{2 }  & 1         & 1 & \blue{ 2-3} & 0          & \blue{2-3}  & \blue{2}   \\
$4_7 $    & 1       & 1       & \red{2}  & \blue{2}   & \blue{3}   & \grey{1-2} & \grey{1-2} & \grey{1-2} & \grey{1-2}  & \blue{2-3}  & 0          & \grey{1-2} \\
$4_8 $    & 1       & 1       & \blue{2} & \blue{2}   & \blue{3}   & 1         & 1         & 1         & \blue{2}    & \blue{2 }   & \grey{1-2}  & 0         \\ \bottomrule
\end{tabular}
\vspace{2mm}
\caption{The $f$-distance table for equivalence classes of knotoids with minimal crossing number $\leq 4$. In a few cases (\emph{e.g} for the pair $(4_1,4_6)$) lower and upper bounds do not coincide. In these cases we write upper and lower bounds separated by a dash, indicating the interval of possible values of the $f$-distances. Entries in the table are colour coded accordingly to how lower bounds were computed. Lower bounds for the entries in blue are computed using the inequality \ref{bound2}, while the ones in red using the inequality \ref{bound1}. We are not able to produce lower bounds for entries in orange.}
\label{tab:cr4}
\end{table}

\normalsize

\begin{table}[H]
\centering
%\rowcolors{2}{gray!25}{white}
\begin{tabular}{@{}llllllllllllll@{}}
\toprule
      & $3_1$&  &       & $3_1$&  &       & $3_1$&  &       & $3_1$&  &        & $3_1$\\ \midrule
$5_1$  & \blue{2}      &  & $6_6$  & \blue{3}      &  & $6_{35}$ & \blue{2}      &  & $6_{64}$ & \blue{3}      &  & $6_{93}$  & \blue{3}      \\
$5_2$  & \blue{2}      &  & $6_7$  & \blue{3}      &  & $6_{36}$ & \blue{2}      &  & $6_{65}$ & \red{2-3}     &  & $6_{94}$  & \blue{2}      \\
$5_3$  & \blue{2}      &  & $6_8$  & \grey{1-2}    &  & $6_{37}$ & \blue{2-4}    &  & $6_{66}$ & \blue{2}      &  & $6_{95}$  & \blue{2}      \\
$5_4$  & \blue{2}      &  & $6_9$  & \grey{1-2}    &  & $6_{38}$ & \blue{2-3}    &  & $6_{67}$ & \blue{2}      &  & $6_{96}$  & \blue{2}      \\
$5_5$  & \blue{3}      &  & $6_{10}$ & \blue{2}      &  & $6_{39}$ & \grey{1-2}    &  & $6_{68}$ & \blue{2}      &  & $6_{97}$  & \blue{2-3}    \\
$5_6$  & \blue{2}      &  & $6_{11}$ & \grey{1-2}    &  & $6_{40}$ & \blue{2}      &  & $6_{69}$ & \blue{2}      &  & $6_{98}$  & \blue{2}      \\
$5_7$  & \blue{3}      &  & $6_{12}$ & \blue{2}      &  & $6_{41}$ & \blue{2}      &  & $6_{70}$ & \blue{3}      &  & $6_{99}$  & \blue{2-4}    \\
$5_8$  & \blue{2}      &  & $6_{13}$ & \blue{2}      &  & $6_{42}$ & \blue{3}      &  & $6_{71}$ & \blue{4}      &  & $6_{100}$ & \blue{2-3}    \\
$5_9$  & \blue{2}      &  & $6_{14}$ & \blue{3}      &  & $6_{43}$ & \blue{3}      &  & $6_{72}$ & \blue{2}      &  & $6_{101}$ & \blue{3}      \\
$5_{10}$ & \blue{3}      &  & $6_{15}$ & 1                              &  & $6_{44}$ & \blue{2}      &  & $6_{73}$ & \blue{2}      &  & $6_{102}$ & \blue{3}      \\
$5_{11}$ & 1                              &  & $6_{16}$ & 1                              &  & $6_{45}$ & \grey{1-2}    &  & $6_{74}$ & \blue{3-4}    &  & $6_{103}$ & \blue{3}      \\
$5_{12}$ & \blue{2}      &  & $6_{17}$ & \blue{2}      &  & $6_{46}$ & \blue{3-4}    &  & $6_{75}$ & \blue{3}      &  & $6_{104}$ & \blue{3}      \\
$5_{13}$ & \blue{4}      &  & $6_{18}$ & \blue{2}      &  & $6_{47}$ & \blue{2}      &  & $6_{76}$ & \blue{3}      &  & $6_{105}$ & \blue{2}      \\
$5_{14}$ & \blue{3}      &  & $6_{19}$ & \blue{2-3}    &  & $6_{48}$ & \blue{2-3}    &  & $6_{77}$ & \blue{2}      &  & $6_{106}$ & 1                              \\
$5_{15}$ & \blue{4}      &  & $6_{20}$ & \blue{3-4}    &  & $6_{49}$ & \blue{3}      &  & $6_{78}$ & \blue{3}      &  & $6_{107}$ & \blue{2}      \\
$5_{16}$ & \red{2}       &  & $6_{21}$ & \red{2}       &  & $6_{50}$ & \blue{2-3}    &  & $6_{79}$ & \blue{3-4}    &  & $6_{108}$ & \blue{2-3}    \\
$5_{17}$ & \blue{2}      &  & $6_{22}$ & \blue{2-3}    &  & $6_{51}$ & \blue{2}      &  & $6_{80}$ & \blue{2-4}    &  & $6_{109}$ & \blue{2-3}    \\
$5_{18}$ & \blue{2}      &  & $6_{23}$ & \blue{2-3}    &  & $6_{52}$ & \blue{3}      &  & $6_{81}$ & \blue{2}      &  & $6_{110}$ & \blue{3}      \\
$5_{19}$ & \blue{2}      &  & $6_{24}$ & \red{2-3}     &  & $6_{53}$ & 1                              &  & $6_{82}$ & \blue{2-3}    &  & $6_{111}$ & \blue{2}      \\
$5_{20}$ & 1                              &  & $6_{25}$ & \blue{3}      &  & $6_{54}$ & \blue{4}      &  & $6_{83}$ & 1-3                            &  & $6_{112}$ & \grey{1-2}    \\
$5_{21}$ & \blue{3}      &  & $6_{26}$ & \blue{2}      &  & $6_{55}$ & \blue{2}      &  & $6_{84}$ & \blue{2}      &  & $6_{113}$ & \blue{3}      \\
$5_{22}$ & \blue{2}      &  & $6_{27}$ & \red{2-3}     &  & $6_{56}$ & \blue{2}      &  & $6_{85}$ & \blue{2-3}    &  & $6_{114}$ & \blue{2}      \\
$5_{23}$ & \blue{3}      &  & $6_{28}$ & \blue{2}      &  & $6_{57}$ & \blue{4}      &  & $6_{86}$ & \blue{2-4}    &  & $6_{115}$ & \grey{1-2}    \\
$5_{24}$ & \grey{1-2}    &  & $6_{29}$ & \blue{2}      &  & $6_{58}$ & \blue{2}      &  & $6_{87}$ & \blue{2}      &  & $6_{116}$ & \blue{2}      \\
$6_1$  & \blue{4}      &  & $6_{30}$ & \grey{1-2}    &  & $6_{59}$ & 1                              &  & $6_{88}$ & 1                              &  & $6_{117}$ & \blue{3}      \\
$6_2$  & \blue{2}      &  & $6_{31}$ & \blue{3}      &  & $6_{60}$ & \blue{3}      &  & $6_{89}$ & \blue{2}      &  & $6_{118}$ & 1-3                            \\
$6_3$  & \blue{2}      &  & $6_{32}$ & \blue{3}      &  & $6_{61}$ & \blue{3}      &  & $6_{90}$ & \red{2}       &  & $6_{119}$ & \blue{2-3}    \\
$6_4$  & \blue{3}      &  & $6_{33}$ & \blue{3}      &  & $6_{62}$ & \grey{1-2}    &  & $6_{91}$ & \red{2-3}    &  & $6_{120}$ & \blue{2-3}    \\
$6_5$  & \blue{2}      &  & $6_{34}$ & \blue{2-3}    &  & $6_{63}$ & \blue{2-3}    &  & $6_{92}$ & \red{2-3}     &  & $6_{121}$ & \blue{2-4}    \\* \bottomrule
\end{tabular}
\vspace{2mm}

\caption{The $f$-distances between equivalence classes of knotoids with minimal crossing number $\leq 6$ and the $3_1$ knotoid. In a few cases (\emph{e.g} for the $6_{99}$ knotoid) lower and upper bounds do not coincide. In these cases we write upper and lower bounds separated by a dash, indicating the interval of possible values of the $f$-distances. Entries in the table are colour coded accordingly to how lower bounds were computed. Lower bounds for the entries in blue are computed using the inequality \ref{bound2}, while the ones in red using the inequality \ref{bound1}. We are not able to produce lower bounds for entries in orange. }
\label{tab:cr6}
\end{table}

\begin{table}[H]
\centering
%\rowcolors{2}{gray!25}{white}
\begin{tabular}{@{}llllllllllllll@{}}
\toprule
& $u_f$&       & $u_f$&         & $u_f$&  & $u_f$&     &$u_f$\\ \midrule
$5_{1}$  & \blue{4}& $5_{2}$  & \blue{2}&$5_{3}$  & \blue{1}&$5_{4}$  & \grey{1-2} &$5_{5}$  & \blue{1}    \\
$5_{6}$  & \blue{1}&$5_{7}$  & \blue{1}&$5_{8}$  & \blue{2}&$5_{9}$  & \blue{1}&$5_{10}$  & \blue{2}    \\
$5_{11}$  & \blue{2}&$5_{12}$  & \blue{3}&$5_{13}$  & \blue{2}&$5_{14}$  & \blue{2}&$5_{15}$  & \blue{2}    \\
$5_{16}$  & \blue{2}&$5_{17}$  & \blue{2}&$5_{18}$  & \blue{1}&$5_{19}$  & \blue{1}&$5_{20}$  & \blue{2}    \\
$5_{21}$  & \blue{1}&$5_{22}$  & \blue{1}&$5_{23}$  & \blue{1}&$5_{24}$  & \blue{2}&$6_{1}$  & \blue{2}    \\
$6_{2}$  & \blue{2}&$6_{3}$  & \blue{2}&$6_{4}$  & \grey{1-2} &$6_{5}$  & \blue{2}&$6_{6}$  & \blue{2}    \\
$6_{7}$  & \blue{3}&$6_{8}$  & \blue{2}&$6_{9}$  & \blue{2}&$6_{10}$  & \grey{1-2} &$6_{11}$  & \blue{2}    \\
$6_{12}$  & \blue{4}& $6_{13}$  & \blue{2}&$6_{14}$  & \blue{2}&$6_{15}$  & \blue{3}&$6_{16}$  & \blue{2}    \\
$6_{17}$  & \grey{1-2} &$6_{18}$  & \blue{2}&$6_{19}$  & \blue{1}&$6_{20}$  & \blue{2}&$6_{21}$  & \blue{3}    \\
$6_{22}$  & \blue{1}&$6_{23}$  & \blue{2}&$6_{24}$  & \blue{2}&$6_{25}$  & \blue{3}&$6_{26}$  & \blue{3}    \\
$6_{27}$  & \blue{1}&$6_{28}$  & \blue{2}&$6_{29}$  & \blue{4}& $6_{30}$  & \blue{2}&$6_{31}$  & \blue{2}    \\
$6_{32}$  & \blue{1}&$6_{33}$  & \blue{3}&$6_{34}$  & \blue{1}&$6_{35}$  & \blue{2}&$6_{36}$  & \grey{1-2}    \\
$6_{37}$  &\grey{1-2} &$6_{38}$  & \blue{2}&$6_{39}$  & \blue{3}&$6_{40}$  & \blue{3}&$6_{41}$  & \grey{1-2}    \\
$6_{42}$  & \blue{2}&$6_{43}$  & \blue{2}&$6_{44}$  & \blue{2}&$6_{45}$  & \blue{2}&$6_{46}$  & \blue{2}    \\
$6_{47}$  & \grey{1-2}&$6_{48}$  & \blue{2} &$6_{49}$  & \blue{1}&$6_{50}$  & \blue{3}&$6_{51}$  & \blue{1}    \\
$6_{52}$  & \blue{2}&$6_{53}$  & \blue{2}&$6_{54}$  & \blue{2}&$6_{55}$  & \blue{2}&$6_{56}$  & \blue{2-3}    \\
$6_{57}$  & \blue{2}&$6_{58}$  & \blue{4}& $6_{59}$  & \blue{2}&$6_{60}$  & \blue{2}&$6_{61}$  & \blue{3}    \\
$6_{62}$  & \blue{2}&$6_{63}$  & \blue{1}&$6_{64}$  & \blue{2}&$6_{65}$  & \blue{2}&$6_{66}$  & \blue{3}    \\
$6_{67}$  & \blue{2}&$6_{68}$  & \blue{2-3} &$6_{69}$  & \blue{2}&$6_{70}$  & \blue{3}&$6_{71}$  & \blue{2}    \\
$6_{72}$  & \blue{3}&$6_{73}$  & \blue{2}&$6_{74}$  & \blue{2-3} &$6_{75}$  & \blue{4}& $6_{76}$  & \blue{3}    \\
$6_{77}$  & \blue{1}&$6_{78}$  & \blue{4}& $6_{79}$  & \blue{2}&$6_{80}$  & \blue{2-3} &$6_{81}$  & \blue{3}    \\
$6_{82}$  & \blue{1}&$6_{83}$  & \blue{2}&$6_{84}$  & \blue{3}&$6_{85}$  & \blue{2}&$6_{86}$  & \blue{2-3}    \\
$6_{87}$  & \blue{2-3} &$6_{88}$  & \blue{2}&$6_{89}$  & \blue{2}&$6_{90}$  & \blue{2}&$6_{91}$  & \blue{3-}    \\
$6_{92}$  & \blue{2-3} &$6_{93}$  & \blue{4}& $6_{94}$  & \blue{3}&$6_{95}$  & \blue{4}& $6_{96}$  & \blue{4}    \\
$6_{97}$  & \grey{1-2}&$6_{98}$  & \blue{3}&$6_{99}$  & \grey{1-2}&$6_{100}$  & \grey{1-2}&$6_{101}$  & \blue{2}    \\
$6_{102}$  & \blue{3}&$6_{103}$  & \blue{5}&$6_{104}$  & \blue{3}&$6_{105}$  & \blue{3}&$6_{106}$  & \blue{3}    \\
$6_{107}$  & \blue{2}&$6_{108}$  & \blue{1}&$6_{109}$  & \blue{1}&$6_{110}$  & \blue{2}&$6_{111}$  & \blue{2}    \\
$6_{112}$  & \blue{2}&$6_{113}$  & \blue{1}&$6_{114}$  & \grey{1-2} &$6_{115}$  & \blue{3}&$6_{116}$  & \blue{2}    \\
$6_{117}$  & \blue{3}&$6_{118}$  & \blue{2-3}&$6_{119}$  & \blue{4}& $6_{120}$  & \blue{2}&$6_{121}$  & \grey{1-2}    \\* \bottomrule

\end{tabular}
\vspace{2mm}

\caption{The unravelling number of knotoids with minimal crossing number $\leq 6$. In these cases we write upper and lower bounds separated by a dash, indicating the interval of possible values of the unravelling number. Entries in the table are colour coded accordingly to how lower bounds were computed. Lower bounds for the entries in blue are computed using the inequality \ref{bound2},while we are not able to produce lower bounds for entries in orange. }
\label{tab:cr6u}
\end{table}

\newpage

\subsection*{Experimental values}\label{appendix:b}

Below, we show Tables \ref{tab:cr4all1}, \ref{tab:cr4all2} and \ref{tab:cr4all3}, containing the experimental $f$-distances between all non-composite knotoid diagrams, including non-minimal crossing representations, with up to six crossings. These have been computed experimentally with the help of a computer program written in \texttt{python 3.7} (see the Main Manuscript for all the details). 

\footnotesize
\begin{table}[H]
\centering
%\rowcolors{2}{gray!25}{white}
\begin{tabular}{@{}lllllllllllllll@{}}
\toprule
&$0_1$   & $2_1$ & $2_1^m$ & $2_1^{ms}$ & $2_1^{s}$ & $3_1$ & $3_1^{m}$ & $3_2$ & $3_2^m$ & $3_2^{ms}$ & $3_2^s$ & $4_1$ & $4_2$ & $4_2^m$ \\ \midrule
$0_1$   & 0    & 1     & 1      & 1     & 1    & 2     & 2    & 1     & 1      & 1     & 1    & 2    & 1     & 1      \\
$2_1$   & 1    & 0     & 2      & 2     & 2    & 1     & 3    & 2     & 2      & 2     & 2    & 3    & 1     & 2      \\
$2_1^m$  & 1    & 2     & 0      & 2     & 2    & 3     & 1    & 2     & 2      & 2     & 2    & 3    & 2     & 1      \\
$2_1^{ms}$ & 1    & 2     & 2      & 0     & 2    & 1     & 3    & 2     & 2      & 2     & 2    & 3    & 2     & 2      \\
$2_1^s$  & 1    & 2     & 2      & 2     & 0    & 3     & 1    & 2     & 2      & 2     & 2    & 3    & 2     & 2      \\
$3_1$   & 2    & 1     & 3      & 1     & 3    & 0     & 4    & 3     & 3      & 3     & 3    & 4    & 2     & 3      \\
$3_1^m$  & 2    & 3     & 1      & 3     & 1    & 4     & 0    & 3     & 3      & 3     & 3    & 4    & 3     & 2      \\
$3_2$   & 1    & 2     & 2      & 2     & 2    & 3     & 3    & 0     & 2      & 2     & 1    & 1    & 2     & 2      \\
$3_2^m$  & 1    & 2     & 2      & 2     & 2    & 3     & 3    & 2     & 0      & 1     & 2    & 1    & 2     & 2      \\
$3_2^{ms}$ & 1    & 2     & 2      & 2     & 2    & 3     & 3    & 2     & 1      & 0     & 2    & 1    & 2     & 2      \\
$3_2^s$  & 1    & 2     & 2      & 2     & 2    & 3     & 3    & 1     & 2      & 2     & 0    & 1    & 2     & 2      \\
$4_1$   & 2    & 3     & 3      & 3     & 3    & 4     & 4    & 1     & 1      & 1     & 1    & 0    & 3     & 3      \\
$4_2$   & 1    & 1     & 2      & 2     & 2    & 2     & 3    & 2     & 2      & 2     & 2    & 3    & 0     & 2      \\
$4_2^m$  & 1    & 2     & 1      & 2     & 2    & 3     & 2    & 2     & 2      & 2     & 2    & 3    & 2     & 0      \\
$4_2^{ms}$ & 1    & 2     & 2      & 1     & 2    & 2     & 3    & 2     & 2      & 2     & 2    & 3    & 2     & 2      \\
$4_2^s$  & 1    & 2     & 2      & 2     & 1    & 3     & 2    & 2     & 2      & 2     & 2    & 3    & 2     & 2      \\
$4_3$   & 2    & 1     & 3      & 3     & 3    & 2     & 4    & 3     & 3      & 3     & 3    & 4    & 2     & 3      \\
$4_3^m$  & 2    & 3     & 1      & 3     & 3    & 4     & 2    & 3     & 3      & 3     & 3    & 4    & 3     & 2      \\
$4_3^{ms}$ & 2    & 3     & 3      & 1     & 3    & 2     & 4    & 3     & 3      & 3     & 3    & 4    & 3     & 3      \\
$4_3^s$  & 2    & 3     & 3      & 3     & 1    & 4     & 2    & 3     & 3      & 3     & 3    & 4    & 3     & 3      \\
$4_4$   & 2    & 2     & 3      & 1     & 3    & 1     & 4    & 3     & 3      & 3     & 3    & 4    & 3     & 3      \\
$4_4^m$  & 2    & 3     & 2      & 3     & 1    & 4     & 1    & 3     & 3      & 3     & 3    & 4    & 3     & 3      \\
$4_4^{ms}$ & 2    & 1     & 3      & 2     & 3    & 1     & 4    & 3     & 3      & 3     & 3    & 4    & 2     & 3      \\
$4_4^s$  & 2    & 3     & 1      & 3     & 2    & 4     & 1    & 3     & 3      & 3     & 3    & 4    & 3     & 2      \\
$4_5$   & 2    & 3     & 1      & 3     & 3    & 4     & 2    & 1     & 3      & 3     & 2    & 2    & 3     & 2      \\
$4_5^m$  & 2    & 1     & 3      & 3     & 3    & 2     & 4    & 3     & 1      & 2     & 3    & 2    & 2     & 3      \\
$4_5^{ms}$ & 2    & 3     & 3      & 3     & 1    & 4     & 2    & 3     & 2      & 1     & 3    & 2    & 3     & 3      \\
$4_5^s$  & 2    & 3     & 3      & 1     & 3    & 2     & 4    & 2     & 3      & 3     & 1    & 2    & 3     & 3      \\
$4_6$   & 3    & 2     & 4      & 2     & 4    & 1     & 5    & 4     & 4      & 4     & 4    & 5    & 3     & 4      \\
$4_6^m$  & 3    & 4     & 2      & 4     & 2    & 5     & 1    & 4     & 4      & 4     & 4    & 5    & 4     & 3      \\
$4_6^{ms}$ & 3    & 2     & 4      & 2     & 4    & 1     & 5    & 4     & 4      & 4     & 4    & 5    & 3     & 4      \\
$4_6^s$  & 3    & 4     & 2      & 4     & 2    & 5     & 1    & 4     & 4      & 4     & 4    & 5    & 4     & 3      \\
$4_7$   & 1    & 2     & 1      & 2     & 2    & 3     & 2    & 2     & 2      & 2     & 2    & 3    & 2     & 2      \\
$4_7^{m}$  & 1    & 1     & 2      & 2     & 2    & 2     & 3    & 2     & 2      & 2     & 2    & 3    & 2     & 2      \\
$4_7^{ms}$ & 1    & 2     & 2      & 2     & 1    & 3     & 2    & 2     & 2      & 2     & 2    & 3    & 2     & 2      \\
$4_7^s$  & 1    & 2     & 2      & 1     & 2    & 2     & 3    & 2     & 2      & 2     & 2    & 3    & 2     & 2      \\
$4_8$   & 1    & 1     & 2      & 2     & 2    & 2     & 3    & 2     & 2      & 2     & 2    & 3    & 2     & 2      \\
$4_8^m$  & 1    & 2     & 1      & 2     & 2    & 3     & 2    & 2     & 2      & 2     & 2    & 3    & 2     & 2      \\
$4_8^{ms}$ & 1    & 2     & 2      & 1     & 2    & 2     & 3    & 2     & 2      & 2     & 2    & 3    & 2     & 2      \\
$4_8^s$  & 1    & 2     & 2      & 2     & 1    & 3     & 2    & 2     & 2      & 2     & 2    & 3    & 2     & 2      \\ \bottomrule
\end{tabular}
\vspace{.2cm}
\caption{Table of experimental $f$-distances of  all knotoids with up to 4 crossings (part 1).}\label{tab:cr4all1}
\end{table}

\begin{table}[H]
\centering
%\rowcolors{2}{gray!25}{white}
\begin{tabular}{@{}lllllllllllllll@{}}
\toprule
& $4_2^{ms}$   & $4_2^s$ & $4_3$ & $4_3^m$ & $4_3^{ms}$ & $4_3^s$ & $4_4$ & $4_4^m$ & $4_4^{ms}$ & $4_4^s$ & $4_5$ & $4_5^m$ & $4_5^{ms}$ & $4_5^s$\\ \midrule
$0_1$   & 1     & 1    & 2     & 2      & 2     & 2    & 2     & 2      & 2     & 2    & 2     & 2      & 2     & 2    \\
$2_1$   & 2     & 2    & 1     & 3      & 3     & 3    & 2     & 3      & 1     & 3    & 3     & 1      & 3     & 3    \\
$2_1^m$  & 2     & 2    & 3     & 1      & 3     & 3    & 3     & 2      & 3     & 1    & 1     & 3      & 3     & 3    \\
$2_1^{ms}$ & 1     & 2    & 3     & 3      & 1     & 3    & 1     & 3      & 2     & 3    & 3     & 3      & 3     & 1    \\
$2_1^s$  & 2     & 1    & 3     & 3      & 3     & 1    & 3     & 1      & 3     & 2    & 3     & 3      & 1     & 3    \\
$3_1$   & 2     & 3    & 2     & 4      & 2     & 4    & 1     & 4      & 1     & 4    & 4     & 2      & 4     & 2    \\
$3_1^m$  & 3     & 2    & 4     & 2      & 4     & 2    & 4     & 1      & 4     & 1    & 2     & 4      & 2     & 4    \\
$3_2$   & 2     & 2    & 3     & 3      & 3     & 3    & 3     & 3      & 3     & 3    & 1     & 3      & 3     & 2    \\
$3_2^m$  & 2     & 2    & 3     & 3      & 3     & 3    & 3     & 3      & 3     & 3    & 3     & 1      & 2     & 3    \\
$3_2^{ms}$ & 2     & 2    & 3     & 3      & 3     & 3    & 3     & 3      & 3     & 3    & 3     & 2      & 1     & 3    \\
$3_2^{s}$  & 2     & 2    & 3     & 3      & 3     & 3    & 3     & 3      & 3     & 3    & 2     & 3      & 3     & 1    \\
$4_1$   & 3     & 3    & 4     & 4      & 4     & 4    & 4     & 4      & 4     & 4    & 2     & 2      & 2     & 2    \\
$4_2$   & 2     & 2    & 2     & 3      & 3     & 3    & 3     & 3      & 2     & 3    & 3     & 2      & 3     & 3    \\
$4_2^m$  & 2     & 2    & 3     & 2      & 3     & 3    & 3     & 3      & 3     & 2    & 2     & 3      & 3     & 3    \\
$4_2^{ms}$ & 0     & 2    & 3     & 3      & 2     & 3    & 2     & 3      & 3     & 3    & 3     & 3      & 3     & 2    \\
$4_2^s$  & 2     & 0    & 3     & 3      & 3     & 2    & 3     & 2      & 3     & 3    & 3     & 3      & 2     & 3    \\
$4_3$   & 3     & 3    & 0     & 4      & 4     & 4    & 3     & 4      & 2     & 4    & 4     & 2      & 4     & 4    \\
$4_3^m$  & 3     & 3    & 4     & 0      & 4     & 4    & 4     & 3      & 4     & 2    & 2     & 4      & 4     & 4    \\
$4_3^{ms}$ & 2     & 3    & 4     & 4      & 0     & 4    & 2     & 4      & 3     & 4    & 4     & 4      & 4     & 2    \\
$4_3^s$  & 3     & 2    & 4     & 4      & 4     & 0    & 4     & 2      & 4     & 3    & 4     & 4      & 2     & 4    \\
$4_4$   & 2     & 3    & 3     & 4      & 2     & 4    & 0     & 4      & 2     & 4    & 4     & 3      & 4     & 2    \\
$4_4^m$  & 3     & 2    & 4     & 3      & 4     & 2    & 4     & 0      & 4     & 2    & 3     & 4      & 2     & 4    \\
$4_4^{ms}$ & 3     & 3    & 2     & 4      & 3     & 4    & 2     & 4      & 0     & 4    & 4     & 2      & 4     & 3    \\
$4_4^s$  & 3     & 3    & 4     & 2      & 4     & 3    & 4     & 2      & 4     & 0    & 2     & 4      & 3     & 4    \\
$4_5$   & 3     & 3    & 4     & 2      & 4     & 4    & 4     & 3      & 4     & 2    & 0     & 4      & 4     & 3    \\
$4_5^m$  & 3     & 3    & 2     & 4      & 4     & 4    & 3     & 4      & 2     & 4    & 4     & 0      & 3     & 4    \\
$4_5^{ms}$ & 3     & 2    & 4     & 4      & 4     & 2    & 4     & 2      & 4     & 3    & 4     & 3      & 0     & 4    \\
$4_5^s$  & 2     & 3    & 4     & 4      & 2     & 4    & 2     & 4      & 3     & 4    & 3     & 4      & 4     & 0    \\
$4_6$   & 3     & 4    & 3     & 5      & 3     & 5    & 2     & 5      & 2     & 5    & 5     & 3      & 5     & 3    \\
$4_6^m$  & 4     & 3    & 5     & 3      & 5     & 3    & 5     & 2      & 5     & 2    & 3     & 5      & 3     & 5    \\
$4_6^{ms}$ & 3     & 4    & 3     & 5      & 3     & 5    & 2     & 5      & 2     & 5    & 5     & 3      & 5     & 3    \\
$4_6^s$  & 4     & 3    & 5     & 3      & 5     & 3    & 5     & 2      & 5     & 2    & 3     & 5      & 3     & 5    \\
$4_7$   & 2     & 2    & 3     & 2      & 3     & 3    & 3     & 3      & 3     & 2    & 2     & 3      & 3     & 3    \\
$4_7^m$  & 2     & 2    & 2     & 3      & 3     & 3    & 3     & 3      & 2     & 3    & 3     & 2      & 3     & 3    \\
$4_7^{ms}$ & 2     & 2    & 3     & 3      & 3     & 2    & 3     & 2      & 3     & 3    & 3     & 3      & 2     & 3    \\
$4_7^s$  & 2     & 2    & 3     & 3      & 2     & 3    & 2     & 3      & 3     & 3    & 3     & 3      & 3     & 2    \\
$4_8$   & 2     & 2    & 2     & 3      & 3     & 3    & 3     & 3      & 2     & 3    & 3     & 2      & 3     & 3    \\
$4_8^m$  & 2     & 2    & 3     & 2      & 3     & 3    & 3     & 3      & 3     & 2    & 2     & 3      & 3     & 3    \\
$4_8^{ms}$ & 2     & 2    & 3     & 3      & 2     & 3    & 2     & 3      & 3     & 3    & 3     & 3      & 3     & 2    \\
$4_8^s$  & 2     & 2    & 3     & 3      & 3     & 2    & 3     & 2      & 3     & 3    & 3     & 3      & 2     & 3    \\ \bottomrule
\end{tabular}
\vspace{.2cm}
\caption{Table of experimental $f$-distances of  all knotoids with up to 4 crossings (part 2).}\label{tab:cr4all2}
\end{table}

% \newpage

\begin{table}[!tbhp]
\centering
%\rowcolors{2}{gray!25}{white}
\begin{tabular}{@{}lllllllllllll@{}}
\toprule
& $4_6$   & $4_6^{m}$ & $4_6^{ms}$ & $4_6^{s}$ & $4_7$ & $4_7^{m}$ & $4_7^{ms}$ & $4_7^{s}$ & $4_8$ & $4_8^{m}$ & $4_8^{ms}$ & $4_8^{s}$  \\ \midrule
$0_1$   & 3     & 3      & 3     & 3    & 1     & 1      & 1     & 1    & 1     & 1      & 1     & 1 \\
$2_1$   & 2     & 4      & 2     & 4    & 2     & 1      & 2     & 2    & 1     & 2      & 2     & 2 \\
$2_1^{m}$  & 4     & 2      & 4     & 2    & 1     & 2      & 2     & 2    & 2     & 1      & 2     & 2 \\
$2_1^{ms}$ & 2     & 4      & 2     & 4    & 2     & 2      & 2     & 1    & 2     & 2      & 1     & 2 \\
$2_1^{s}$  & 4     & 2      & 4     & 2    & 2     & 2      & 1     & 2    & 2     & 2      & 2     & 1 \\
$3_1$   & 1     & 5      & 1     & 5    & 3     & 2      & 3     & 2    & 2     & 3      & 2     & 3 \\
$3_1^{m}$  & 5     & 1      & 5     & 1    & 2     & 3      & 2     & 3    & 3     & 2      & 3     & 2 \\
$3_2$   & 4     & 4      & 4     & 4    & 2     & 2      & 2     & 2    & 2     & 2      & 2     & 2 \\
$3_2^{m}$  & 4     & 4      & 4     & 4    & 2     & 2      & 2     & 2    & 2     & 2      & 2     & 2 \\
$3_2^{ms}$ & 4     & 4      & 4     & 4    & 2     & 2      & 2     & 2    & 2     & 2      & 2     & 2 \\
$3_2^{s}$  & 4     & 4      & 4     & 4    & 2     & 2      & 2     & 2    & 2     & 2      & 2     & 2 \\
$4_1$   & 5     & 5      & 5     & 5    & 3     & 3      & 3     & 3    & 3     & 3      & 3     & 3 \\
$4_2$   & 3     & 4      & 3     & 4    & 2     & 2      & 2     & 2    & 2     & 2      & 2     & 2 \\
$4_2^{m}$  & 4     & 3      & 4     & 3    & 2     & 2      & 2     & 2    & 2     & 2      & 2     & 2 \\
$4_2^{ms}$ & 3     & 4      & 3     & 4    & 2     & 2      & 2     & 2    & 2     & 2      & 2     & 2 \\
$4_2^{s}$  & 4     & 3      & 4     & 3    & 2     & 2      & 2     & 2    & 2     & 2      & 2     & 2 \\
$4_3$   & 3     & 5      & 3     & 5    & 3     & 2      & 3     & 3    & 2     & 3      & 3     & 3 \\
$4_3^{m}$  & 5     & 3      & 5     & 3    & 2     & 3      & 3     & 3    & 3     & 2      & 3     & 3 \\
$4_3^{ms}$ & 3     & 5      & 3     & 5    & 3     & 3      & 3     & 2    & 3     & 3      & 2     & 3 \\
$4_3^{s}$  & 5     & 3      & 5     & 3    & 3     & 3      & 2     & 3    & 3     & 3      & 3     & 2 \\
$4_4$   & 2     & 5      & 2     & 5    & 3     & 3      & 3     & 2    & 3     & 3      & 2     & 3 \\
$4_4^{m}$  & 5     & 2      & 5     & 2    & 3     & 3      & 2     & 3    & 3     & 3      & 3     & 2 \\
$4_4^{ms}$ & 2     & 5      & 2     & 5    & 3     & 2      & 3     & 3    & 2     & 3      & 3     & 3 \\
$4_4^{s}$  & 5     & 2      & 5     & 2    & 2     & 3      & 3     & 3    & 3     & 2      & 3     & 3 \\
$4_5$   & 5     & 3      & 5     & 3    & 2     & 3      & 3     & 3    & 3     & 2      & 3     & 3 \\
$4_5^{m}$  & 3     & 5      & 3     & 5    & 3     & 2      & 3     & 3    & 2     & 3      & 3     & 3 \\
$4_5^{ms}$ & 5     & 3      & 5     & 3    & 3     & 3      & 2     & 3    & 3     & 3      & 3     & 2 \\
$4_5^{s}$  & 3     & 5      & 3     & 5    & 3     & 3      & 3     & 2    & 3     & 3      & 2     & 3 \\
$4_6$   & 0     & 6      & 2     & 6    & 4     & 3      & 4     & 3    & 3     & 4      & 3     & 4 \\
$4_6^{m}$  & 6     & 0      & 6     & 2    & 3     & 4      & 3     & 4    & 4     & 3      & 4     & 3 \\
$4_6^{ms}$ & 2     & 6      & 0     & 6    & 4     & 3      & 4     & 3    & 3     & 4      & 3     & 4 \\
$4_6^{s}$  & 6     & 2      & 6     & 0    & 3     & 4      & 3     & 4    & 4     & 3      & 4     & 3 \\
$4_7$   & 4     & 3      & 4     & 3    & 0     & 2      & 2     & 2    & 2     & 2      & 2     & 2 \\
$4_7^{m}$  & 3     & 4      & 3     & 4    & 2     & 0      & 2     & 2    & 2     & 2      & 2     & 2 \\
$4_7^{ms}$ & 4     & 3      & 4     & 3    & 2     & 2      & 0     & 2    & 2     & 2      & 2     & 2 \\
$4_7^{s}$  & 3     & 4      & 3     & 4    & 2     & 2      & 2     & 0    & 2     & 2      & 2     & 2 \\
$4_8$   & 3     & 4      & 3     & 4    & 2     & 2      & 2     & 2    & 0     & 2      & 2     & 2 \\
$4_8^{m}$  & 4     & 3      & 4     & 3    & 2     & 2      & 2     & 2    & 2     & 0      & 2     & 2 \\
$4_8^{ms}$ & 3     & 4      & 3     & 4    & 2     & 2      & 2     & 2    & 2     & 2      & 0     & 2 \\
$4_8^{s}$  & 4     & 3      & 4     & 3    & 2     & 2      & 2     & 2    & 2     & 2      & 2     & 0 \\ \bottomrule
\end{tabular}
\vspace{.2cm}
\caption{Table of experimental $f$-distances of  all knotoids with up to 4 crossings (part 3).}\label{tab:cr4all3}
\end{table}

\newpage

\normalsize

\begin{rmk}\label{rem:pca}
By employing statistical procedures we can re-arrange the table of numerical $f$-distances so that the isotopy classes of knotoids are ordered with respect to their proximity. More precisely, by considering Tables \ref{tab:cr4all1}-\ref{tab:cr4all3} as a $40\times40$ matrix $M$ and then shifting its empirical mean to zero, we can apply Principal Component Analysis (PCA) to move the data to a new orthogonal coordinate system where the greatest variance of the data appears by projecting along the first coordinate. This corresponds to the eigenvector that is related to the highest eigenvalue of the correlation matrix $\frac{1}{n-1}M^T M$, where $n$ is the number of rows of $M$. The post-PCA matrix provides a more comprehensive overview of the knotoids space and allows for an easier exploration of potential relations between knotoids. A graphical representation of the results of applying PCA on the set of knotoids is shown below in Figure \ref{fig:graph_pca}. The analysis and the graphic representations below were done using the statistical package \texttt{R}.

\begin{figure}[!tbhp]

 \includegraphics[scale=.7]{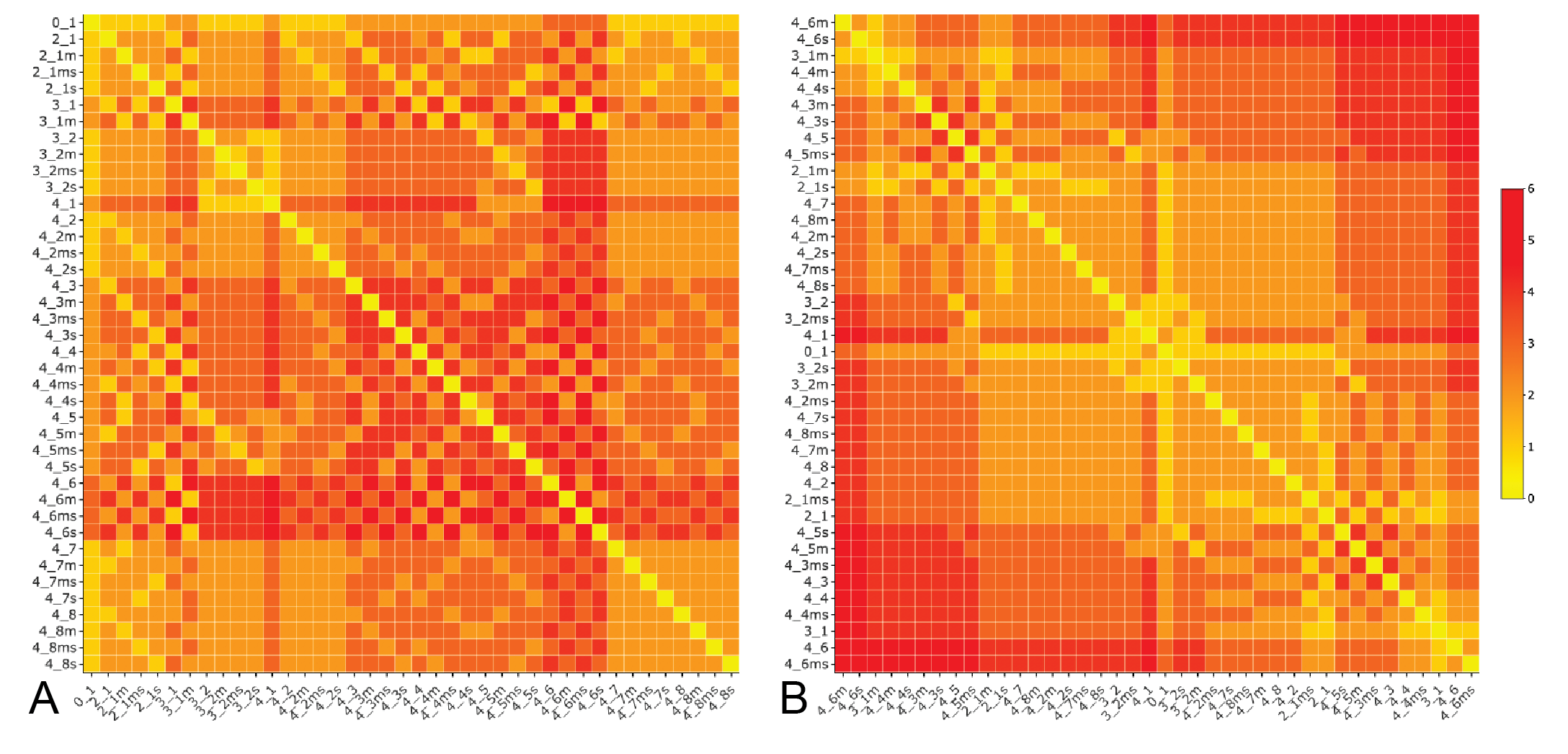}
 \caption{\textbf{Graphical representation of experimental values.} A graphical representation of  the set of all knotoids with up to 4 crossings (on the left) before and after (on the right) using PCA. On the right, the trivial knotoid is placed in the middle of the figure since while each knotoid together with its rotation (here denoted as \texttt{ms}) are always on the opposite side of $0_1$ and in equal distance from it  than its mirror reflection and its symmetric involutions. Note also that the same holds for $4_1$ due to its amphichirality and since two knotoids cannot occupy the same spot, it is place immediately to the left of $0_1$. The legend on the far right shows the correspondence between distance and colour.} 
 \label{fig:graph_pca}
 \end{figure}
  \end{rmk}

\end{document}